\documentclass[10pt]{amsart}
\usepackage{amssymb}
\usepackage{amsfonts}
\usepackage{latexsym}
\usepackage{amscd}

\input xy
\xyoption{all}

\numberwithin{equation}{section}

\usepackage{enumerate}

\newcommand{\Z}{{\mathbb{Z}}}

\def\GL{\mathop{\hbox{GL}}}
\def\soc{\mathop{\hbox{Soc}}}

\def\mt3{\mathop{\hbox{MT3}}}

\def\Z{\mathbb Z}

\def\Sh{\mathop{\hbox{Sh}}}
\def\II{\mathbf{II}}
\def\I{\mathbf{I}}

\newtheorem{theorem}{Theorem}[section]

\newtheorem{corollary}[theorem]{Corollary}

\newtheorem{definition}[theorem]{Definition}

\newtheorem{lemma}[theorem]{Lemma}

\newtheorem{proposition}[theorem]{Proposition}
\newtheorem{remark}[theorem]{Remark}

\usepackage[all,poly,graph]{xy}

\title{Atlas of Leavitt Path Algebras of small graphs}

\author[P. Alberca]{P. Alberca Bjerregaard}
\address{P. Alberca Bjerregaard and D. Mart\'{\i}n Barquero: Departamento de Matem\'atica Aplicada,
Universidad de M\'alaga, 29071
M\'alaga, Spain.}
\email{pgalberca@uma.es, dmartin@uma.es}
\author[G. Aranda]{G. Aranda Pino}
\address{G. Aranda Pino, C. Mart\'\i n Gonz\'alez and M. Siles Molina: Departamento de \'Algebra,
Geometr\'\i a y
Topolog\'\i a, Universidad de M\'alaga, 29071 M{\'a}laga, Spain} \email{g.aranda@uma.es,
candido@apncs.cie.uma.es, msilesm@uma.es}
\author[D. Mart\'{\i}n]{D. Mart\'{\i}n Barquero}
\author[C. Mart\'{\i}n]{C. Mart\'{\i}n Gonz\'alez}
\author[M. Siles]{M. Siles Molina}
\thanks{The authors have been supported by the Spanish
MEC and Fondos FEDER through project MTM2007-60333, jointly by the Junta de Andaluc\'{\i}a and Fondos FEDER
through projects FQM-336, FQM-2467 and FQM-3737 and by the Spanish Ministry of Education and Science under
project ``Ingenio Mathematica (i-math)'' No. CSD2006-00032 (Consolider-Ingenio 2010).}

\begin{document}

\maketitle
\begin{abstract}
The aim of this work is the description of the isomorphism classes of all Leavitt path algebras coming from
graphs satisfying Condition (Sing) with up to three vertices. In particular, this classification recovers the
one achieved by Abrams et al. \cite{AALP} in the case of graphs whose Leavitt path algebras are purely
infinite simple. The description of the isomorphism classes is given in terms of a series of invariants
including the $\bf{K}_0$ group, the socle, the number of loops with no exits and the number of hereditary and
saturated subsets of the graph.

\end{abstract}

\section*{Introduction}

For a graph $E$ and field $K$, the Leavitt path algebras $L_K(E)$ can be regarded as both a broad
generalization of the algebras constructed by W. G. Leavitt in \cite{Leavitt, Le} to produce rings that do not
satisfy the IBN property, and as the algebraic siblings of the graph C*-algebras $C^*(E)$ \cite{BPRS, R},
which in turn are the analytic counterpart and descendant from the algebras investigated by J. Cuntz in
\cite{Cuntz, CK}.

\medskip

The first appearance of $L_K(E)$ took place in the papers \cite{AA1} and \cite{AMP}, in the context of
row-finite graphs (countable graphs such that every vertex emits only a finite number of edges). Although
their history is very recent, a flurry of activity has followed since the beginning of the theory, in several
different directions: characterization of algebraic properties of $L_K(E)$ in terms of graph properties of $E$
(see for instance \cite{AA1, AA2, AAPS, APS}); study of the modules over $L_K(E)$ in \cite{AB, AMMS1} among
others; computation of various substructures such as the Jacobson radical, the center, the socle and the
singular ideal in \cite{AA3, AC, AMMS1, S} respectively; investigation of the relationship and connections
with $C^*(E)$ and general C*-algebras \cite{AbramsTomforde, AMP, AGPS, APS2}; generalization to countable but
not necessarily row-finite graphs in \cite{AA3, AMMS2, Tomforde}, and then for completely arbitrary graphs in
\cite{AR, ARS, ArRS, G}; K-Theory \cite{AB, ABC, AMP}; and classification programs \cite{AALP, ALPS}.

\medskip

This last line of research is the main concern of this paper.
 Concretely, we classify Leavitt path algebras of
graphs of up to three vertices without parallel edges or, in a more standard terminology, graphs
satisfying Condition (Sing). Given the particular nature of our task, we employ a taxonomic modus operandi 
which some people would associate with biology rather than mathematics. Thus, in order to achieve our goal, we will
apply several known invariants for Leavitt path algebras (i.e., properties or structures that are preserved by
\emph{ring} isomorphisms between Leavitt path algebras) as well as find and prove some other completely new,
thus contributing as a byproduct to finding further characterizations and relations of algebraic properties of
$L_K(E)$ with graph-theoretic properties of $E$.

\medskip

In particular, our classification allows to recover the result that Abrams et al. \cite[Proposition 4.2]{AALP}
in which they showed that the information on the ${\bf K}_0$ groups and unit $[1_{L_K(E)}]$ is enough to
classify purely infinite simple unital Leavitt path algebras. We completely remove the condition of being
``purely infinite simple" and find a set of invariants (now including more that merely the basic ${\bf
K}$-theory data) that can distinguish any two Leavitt path algebras of the graphs within our scope, building
in this way the ``atlas of Leavitt path algebras of small graphs".

\medskip

The reason why both \cite[Proposition 4.2]{AALP} and our results in this article (Theorems \ref{casethree} and
\ref{caselessthanthree}) focus on the family $\{E\ |\ E \text{ has Condition (Sing) and }|E^0|\leq 3\}$ are
natural: on the one hand it was proved in \cite[Proposition 3.4]{AALP} that every purely infinite simple
Leavitt path algebra is isomorphic to some other having an underlying graph that satisfies Condition (Sing)
(actually this result can be carried over for not necessarily purely infinite Leavitt path algebras if we
forget about some conditions that are not needed for our purposes, such as the cardinals of the sets of
edges). Thus, in order to classify \emph{all} the Leavitt path algebras, it is enough to classify those
generated by graphs satisfying Condition (Sing).

\medskip

Moreover, in the enterprise of completing an atlas for Leavitt path algebras, the Condition (Sing) is
compulsory, because as soon as we allow arbitrary parallel edges in our graphs, we obtain infinite families of
non-isomorphic Leavitt path algebras. Indeed, for any $n\in {\mathbb N}$ the graph

\medskip

\vskip-0.7cm
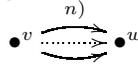
$$\xymatrix{{\bullet}^{v} \ar@/_0.5pc/ [r] \ar@{.>} [r] \ar@/^0.5pc/[r]^{n)} & {\bullet}^{w} }$$
\vskip-0.1cm

\medskip

\noindent
is such that $L_K(E_n)\cong {\mathcal M}_n(K)$ and hence $\{L_K(E_n)\ |\ n\in {\mathbb N}\}$ is an infinite
family of mutually non-isomorphic Leavitt path algebras of graphs of order two.

\medskip

In the current state of the art concerning the classification of Leavitt path algebras, the condition that $|E^0|\leq 3$ is necessary.
If we think of the case $n=4$, for which there would be $3044$ graphs to be studied, even though the classification would still be tractable 
from a computational point of view (because of the ``not so large" size), the difficulty arises because it is not clear which collection of invariants will be fine enough  to get this desired classification. To enlighten this statement, we refer the reader to \cite{CMM}, where a first approach to this problem is tackled and where the authors explain which are the difficulties to get the classification in the case $n=4$. Note that they restrict their attention to those Leavitt path algebras which are simple.

\medskip

The way to proceed will be to use a matrix approach based on adjacency matrices (graphs satisfying Condition
(Sing) have binary adjacency matrices, that is, matrices with entries in the set $\{0,1\}$).  The abundance of
properties of $L_K(E)$ which can be investigated directly in the graph $E$ (or equivalently in its adjacency
matrix) together with the fact that matrices can be handled with computational techniques, imply that matrix
methods can be successfully exploited in the classification of Leavitt path algebras.

\medskip

One of the drawbacks of the adjacency matrix approach is that different matrices can represent the same graph
(up to relabeling of vertices): if a matrix $B$ can be obtained from a matrix $A$ by a series of
(simultaneous) permutations of rows and columns, then $A$ and $B$ represent isomorphic graphs, so first we
have the problem of classifying orbits of the action of the symmetric group $S_n$ on the set of binary
$n\times n$ matrices.

\medskip

Once this has been done, further computational tools are applied to eliminate matrices which agree after a
shift process (it is known \cite[Theorem 3.11]{ALPS} that shift graphs produce isomorphic Leavitt path
algebras). Thus, after taking one representative of each orbit (under the action of $S_n$) and eliminating
coincident matrices (up to shift process), we get a restricted list of matrices that represent the graphs of
the Leavitt path algebras that must be classified.

In order to do that, we set up a list of invariants. Some of them are well-known, such as the ${\bf K}_0$
groups, the socle, the units $[1_{L_K(E)}]$, etc.; and some of them have been found, proved, and tailored here
specifically for our purposes, such as the the number of hereditary and saturated subsets of vertices, the
number of isolated loops, the quotient modulo the only nontrivial hereditary and saturated subset (when this
is the case), etc.

\medskip

In any case, even graph-theoretic data in the table also have an algebraic nature: ILN characterizes the
number of ideals generated by idempotents that are isomorphic to $K[x,x^{-1}]$, HS is the number of ideals
generated by idempotents and MT3+L characterizes primitivity. The reason to include these graph-theoretic
invariants in the tables rather than their algebraic equivalents, is because the first ones are easily
recognized and computed for any given graph.

\medskip

For all the computations we have implemented and used pieces of \emph{Magma} and \emph{Mathematica} codes,
which we list in the Appendix. Specifically, and for optimization reasons, the computation of the invariants
has been performed by the \emph{Mathematica} software, whereas for the calculation of the orbits and shift
graphs the \emph{Magma} software has been used instead, as it proved to be faster and more efficient for these
purposes. The reading of these codes can be of interest in order to learn how the $K_0$ group is computed as
well as the process by which some redundant graphs (i.e., those that already belong to some existing orbit and
also those that appear as shifts of some other, as explained before) have been eliminated and do not appear in
the tables.

\medskip

\section {Definitions}

In this section we collect various notions concerning graphs, after which we define Leavitt path algebras.

\medskip

A (\emph{directed}) \emph{graph} $E=(E^0,E^1,r,s)$ consists of two sets $E^0$ and $E^1$ together with maps
$r,s:E^1 \to E^0$. The elements of $E^0$ are called \emph{vertices} and the elements of $E^1$ \emph{edges}.
For $e\in E^1$, the vertices $s(e)$ and $r(e)$ are called the \emph{source} and \emph{range} of $e$,
respectively, and $e$ is said to be an \emph{edge from $s(e)$ to $r(e)$}. If $s^{-1}(v)$ is a finite set for
every $v\in E^0$, then the graph is called \emph{row-finite}.

\medskip

If $E^0$ is finite and $E$ is row-finite, then $E^1$ must necessarily be finite as well; in this case we say
simply that $E$ is \emph{finite}. Even though many of the results of the paper hold for not necessarily
finite or row-finite graphs, we will assume that our graphs are finite, unless otherwise noted. By
\emph{order} of a finite graph $E$ we will understand the cardinal of $E^0$. In what follows, for any set
$X$, we will denote the cardinal of $X$ by $\vert X\vert$.

\medskip

A vertex which emits no edges is called a \emph{sink}. A \emph{path} $\mu$ in a graph $E$ is a finite sequence
of edges $\mu=e_1\dots e_n$ such that $r(e_i)=s(e_{i+1})$ for $i=1,\dots,n-1$. In this case, $s(\mu)=s(e_1)$
and $r(\mu)=r(e_n)$ are the \emph{source} and \emph{range} of $\mu$, respectively, and $n$ is the
\emph{length} of $\mu$, denoted by $l(\mu)$. We view the elements of $E^{0}$ as paths of length
$0$. Define ${\rm{Path}}(E)$ to be the set of all paths.

\medskip

If $\mu$ is a path in $E$, and if $v=s(\mu)=r(\mu)$, then $\mu$ is called a \emph{closed path based at $v$}.
If $s(\mu)=r(\mu)$ and $s(e_i)\neq s(e_j)$ for every $i\neq j$, then $\mu$ is called a \emph{cycle}. A graph
which contains no cycles is called \emph{acyclic}.

\medskip

An edge $e$ is an {\it exit} for a path $\mu = e_1 \dots e_n$ if there exists $i$ such that $s(e)=s(e_i)$ and
$e\neq e_i$. We say that $E$ satisfies \emph{Condition} (L) if every cycle in $E$ has an exit.

\medskip

We define a relation $\ge$ on $E^0$ by setting $v\ge w$ if there exists a path in $E$ from $v$ to $w$. In this
situation we say that $v$ \emph{connects} to $w$. A subset $H$ of $E^0$ is called \emph{hereditary} if $v\ge
w$ and $v\in H$ imply $w\in H$. A hereditary set is \emph{saturated} if every regular vertex which feeds into
$H$ and only into $H$ is again in $H$, that is, if $s^{-1}(v)\neq \emptyset$ is finite and
$r(s^{-1}(v))\subseteq H$ imply $v\in H$. Denote by $\mathcal{H}_E$ the set of hereditary saturated subsets of
$E^0$.

\medskip

The set $T(v)=\{w\in E^0\mid v\ge w\}$ is the \emph{tree} of $v$, and it is the smallest hereditary subset of
$E^0$  containing $v$. We extend this definition for an arbitrary set $X\subseteq E^0$ by $T(X)=\bigcup_{x\in
X} T(x)$. The \emph{hereditary saturated closure} of a set $X$ is defined as the smallest hereditary and
saturated subset of $E^0$ containing $X$. It is shown in \cite{AMP, BHRS} that the hereditary saturated
closure of a set $X$ is $\overline{X}=\bigcup_{n=0}^\infty \Lambda_n(X)$, where

\medskip

\begin{enumerate}

\item[] $\Lambda_0(X)=T(X)$, and
\item[] $\Lambda_n(X)=\{y\in E^0\mid
s^{-1}(y)\neq \emptyset$ and $r(s^{-1}(y))\subseteq \Lambda_{n-1}(X)\}\cup \Lambda_{n-1}(X)$, for $n\ge 1$.
\end{enumerate}

\medskip

Let $K$ be an arbitrary field and $E$ be a row-finite graph. The {\em Leavitt path $K$-algebra} $L_K(E)$ is
defined to be the $K$-algebra generated by the set $E^0\cup E^1\cup \{e^*\mid e\in E^1\}$ with the following
relations:

\medskip

\begin{enumerate}
\item[\rm{(V)}] $vw= \delta_{v,w}v$ for all $v,w\in E^0$.
\item[\rm{(E1)}]  $s(e)e=er(e)=e$ for all $e\in E^1$.
\item[\rm{(E2)}]  $r(e)e^*=e^*s(e)=e^*$ for all $e\in E^1$.
\item[\rm{(CK1)}]  $e^*f=\delta _{e,f}r(e)$ for all $e,f\in E^1$.
\item[\rm{(CK2)}]  $v=\sum _{e\in s^{-1}(v)}ee^*$ for every $v\in E^0$ that is not a sink.
\end{enumerate}

\medskip

Relation (V) is related to vertices, (E1) and (E2) refer to edges, while the names Cuntz and Krieger give rise
to the letters which comprise the notation (CK1) and (CK2) (notation which is now standard in both the
algebraic and the analytic literature).

\medskip

The elements of $E^1$ are called \emph{real edges}, while for $e\in E^1$ we call $e^\ast$ a \emph{ghost edge}.
The set $\{e^*\mid e\in E^1\}$ will be denoted by $(E^1)^*$.  We let $r(e^*)$ denote $s(e)$, and we let
$s(e^*)$ denote $r(e)$.  If $\mu = e_1 \dots e_n$ is a path in $E$, { we write $\mu^*$ for} the element $e_n^*
\dots e_1^*$ of $L_{K}(E)$. For any subset $H$ of $E^0$, we will denote by $I(H)$ the ideal of $L_{K}(E)$
generated by $H$. Note that if $E$ is a finite graph, then $L_{K}(E)$ is unital with $\sum _{v\in E^0}
v=1_{L_{K}(E)}$; otherwise, $L_{K}(E)$ is a ring with a set of local units consisting of sums of distinct
vertices.

\medskip

The Leavitt path algebra $L_{K}(E)$ is a $\mathbb{Z}$-graded $K$-algebra, spanned as a $K$-vector space by
$\{pq^{\ast } \ \vert \ p,q \in {\rm{Path}}(E)\}$. (Recall that the elements of $E^{0} $ are viewed
as paths of length $0$, so that this set includes elements of the form $v$ with $v\in E^{0}$.) In particular,
for each $n\in\mathbb{Z}$, the degree $n$ component $L_{K}(E)_{n}$ is spanned by
$\{pq^{\ast }\ \vert \ p,q \in {\rm{Path}}(E), \  l(p)-l(q)=n\}$.

\medskip

For a hereditary subset $H$ of $E^0$, the \emph{quotient graph} $E/H$ is defined as
$$(E^0\setminus H, \{e\in E^1|\ r(e)\not\in H\}, r|_{(E/H)^1}, s|_{(E/H)^1}),$$
\noindent
and \cite[Lemma 2.3 (1)]{APS}
shows that if $H$ is hereditary and saturated, then $L_K(E)/I(H)\cong L_K(E/H)$, isomorphism of ${\mathbb
Z}$-graded $K$-algebras.

\medskip

Given a graph $E$, the \emph{adjacency matrix} is the matrix $A_E=(a_{ij}) \in {\mathbb Z}^{(E^0\times E^0)}$,
given by $a_{ij} = \vert\{\text{edges from }i\text{ to }j\}\vert$.

\medskip

Even though Leavitt path algebras are ${\mathbb Z}$-graded $K$-algebras with involution $*$, all our
homomorphisms and isomorphism will be \emph{ring} morphisms (not necessarily graded morphisms, or algebra
morphisms, or $*$-morphisms). In particular when we say that a property (P) is an \emph{invariant} for Leavitt
path algebras we mean that if a graph $E$ satisfies (P) and there exists a \emph{ring} isomorphism
$f:L_K(E)\to L_K(F)$, then $F$ necessarily satisfies (P). For more on the subtleties regarding the differences
and connections between ring, algebra, and *-algebra isomorphisms between $L_K(E)$ and $L_K(F)$, we refer the
reader to \cite{AbramsTomforde}.

\medskip

\section{Matrix techniques}
A useful way to work with finite order graphs is to consider their adjacency matrices. Consider for instance
the graphs

\medskip

$$\xymatrix{ {\bullet}^2 & {\bullet}^1 \ar[r]\ar[l] & {\bullet}^3 & \text{  } &
{\bullet}^1 & {\bullet}^2 \ar[r]\ar[l] & {\bullet}^3}$$

\medskip

\noindent whose adjacency matrices are $\tiny\begin{pmatrix}0 & 1 & 1\cr 0 & 0 & 0 \cr 0 & 0 & 0\end{pmatrix}$
and $\tiny\begin{pmatrix}0 & 0 & 0\cr 1 & 0 & 1 \cr 0 & 0 & 0\end{pmatrix}$, respectively. The two graphs are
essentially the same (i.e., they are isomorphic graphs) although the matrices are different. It is easy to
prove that when we permute two vertices in a graph, the corresponding adjacency matrices are related by a
composition of permutations of rows and columns (so they are similar matrices). In the example above the
second matrix is obtained by permuting rows and columns $1$ and $2$ of the first matrix.

\medskip

If we have a graph $E$ with vertices labeled $\{1,2,\ldots, n\}$ and permute labels $i$ and $j$ we get a new
graph $E'$. Then, denoting by $M$ and $M'$ the corresponding adjacency matrices we may relate them as follows:
consider the $n\times n$ integer matrix $e_{ij}$ with all entries $0$ except for the $(i,j)$ one which is $1$.
Consider also, for $i\ne j$, the matrix $I_{ij}:=1-e_{ii}-e_{jj}+e_{ij}+e_{ji}$, that is, the identity matrix
with rows $i$ and $j$ permuted. We have $I_{ij}^2=1$ so that $I_{ij}\in\hbox{GL}_n (\mathbb Z)$. As it is well
known, for any matrix $M$ the new matrix $M'=I_{ij}MI_{ij}$ agrees with $M$ except for the fact that  rows and
columns $i$ and $j$ of $M$ are permuted in the new matrix.

\medskip

Since $E$ and $E'$ are isomorphic graphs, the matrices $M$ and $M'$ represent the same graph. In other words,
the problem of classifying graphs (up to isomorphism) of a given order is equivalent to that of studying the
orbits of the subgroup $\langle I_{ij}\colon i\ne j\rangle\leq \hbox{GL}_n(\mathbb Z)$ on $\mathcal
M_n(\mathbb Z)$ by the usual conjugation action.

\medskip

On the other hand it is easy to check that the map $  \langle I_{ij}\colon i\ne j\rangle\to S_n$ given by
$I_{ij}\mapsto (ij)$ is a group isomorphism from our group of matrices to the symmetric group of permutations
of $\{1,\ldots,n\}$, where $(ij)$ denotes the permutation of elements $i$ and $j$.

\medskip

In other words, we are concerned with the problem of studying the action of the symmetric group $S_n$ on the
set of binary $n\times n$ matrices, that is, on the set $\mathcal M_n(\Z_2)$ which has cardinal $2^{n^2}$. To
obtain some additional information on the complexity of this problem we recall some basic results on actions
of finite groups $G$ on finite sets $X$. These are given by maps $G\times X\to X$ in which the action of $g\in
G$ on $x\in X$ is denoted by $g x$. Let us denote by $X/G$ the set of orbits of $X$ under the action of the
group $G$. Then, as it is well known,

\medskip

\begin{equation}\label{noo}
\tag{\ensuremath{\dagger}} \vert X/G\vert=\frac{1}{\vert G\vert}\sum_{g\in G} \vert X_g\vert, \quad \text{where $X_g:=\{x\in X\colon g x=x\}$.}
\end{equation}

\medskip

\begin{proposition}\label{paluego}
Denote by $\Phi_n$ the number of non-isomorphic graphs of order $n$ which satisfy Condition {\rm (Sing)}. Then
$\Phi_1=2$, $\Phi_2=10$, $\Phi_3=104$ and $\Phi_4=3044$.
\end{proposition}
\begin{proof} The case $n=1$ is trivial. For the case $n=2$ we need to calculate the number of orbits of
$S_2=\{1,(12)\}$ on the set $X=\mathcal M_2(\Z_2)$. In this case $X_1=X$ so that $\vert X_1\vert=2^4$ while
$X_{(12)}$ is the set of matrices  of the form {\tiny$\begin{pmatrix}a & b\cr b & a\end{pmatrix}$}, which is a
$\Z_2$-vector space of dimension $2$ hence has cardinal $\vert X_{(12)}\vert=2^2$. Therefore the number of
non-isomorphic graphs of order $2$ is $\vert X/S_2\vert=\frac{1}{2}(2^4+2^2)=10$.

\medskip

Let us consider $n=3$ now. We have that $\Phi_3=\vert \mathcal M_3(\Z_2)/S_3\vert$ so we must investigate the
summands $X_g$ in formula (\ref{noo}), for $g\in S_3$. It is worth to realize that in the formula (\ref{noo})
we have $\vert X_g\vert=\vert X_h\vert$ if $g$ and $h$ are conjugated. Since
$S_3=\{1,(12),(13),(23),(123),(132)\}$ and the conjugacy classes in $S_3$ are $\{1\}$, $\{(12),(13),(23)\}$
and  $\{(123),(132)\}$, we have $\Phi_3=\frac{1}{6}( \vert X_1\vert+3\vert X_{(12)}\vert+2\vert
X_{(123)}\vert)$. On the other hand the matrices fixed by $(12)$ are those of the form
$\begin{pmatrix}
a & b & c\cr
b & a & c\cr
d & d & e
 \end{pmatrix}$
with $a,b,c,d,e\in {\mathbb Z}_2$. These constitute a vector space $X_{(12)}$ of dimension $5$, hence $\vert
X_{(12)}\vert=2^5$. The matrices fixed by $(123)$ are those of the form
$\begin{pmatrix}
a & b & c\cr
c & a & b\cr
b & c & a
 \end{pmatrix}$
with $a,b,c\in {\mathbb Z}_2$. In this case the vector space $X_{(123)}$ has dimension $3$ and therefore
$\vert X_{(123)}\vert=2^3$. Thus
$\Phi_3=\frac{1}{3}(2^9+3\cdot 2^5+2\cdot 2^3)=\frac{512+96+16}{6}=104.$

The computations for $S_4$ and $X=\mathcal M_4(\Z_2)$ are as follows:
there are five conjugacy classes on $S_4$ which are
\begin{itemize}
\item $\{1\}$
\item $\{(12), (13), (14), (23), (24), (34)\}$,
\item $\{(123), (132), (124), (142), (134), (143), (234), (243)\}$,
\item $\{(12)(34), (13)(24), (14)(23)\}$,
\item $\{(1234), (1243), (1324), (1342), (1423), (1432)\}$.
\end{itemize}

\medskip

Therefore $\Phi_4=\frac{1}{24}(X_1+
6 X_{(12)}+8 X_{(123)}+3 X_{(12)(34)}+6 X_{(1234)})$. Then $\vert X_1\vert=\vert X\vert=2^{16}$.
Moreover $X_{(12)}$, $X_{(123)}$, $X_{(12)(34}$ and $X_{(1234)}$ are (respectively) the sets of matrices:

\medskip

{\small
$$\begin{pmatrix}a & b & c & d\cr b & a & c & d\cr e & e & x & y\cr f & f & u &v\end{pmatrix},
\begin{pmatrix}\lambda & a & b & z\cr b & \lambda & a & z\cr a & b & \lambda & z\cr t & t & t &\mu\end{pmatrix},
\begin{pmatrix}a & b & x & y\cr b & a & y & x\cr x' & y'& c & d\cr y' & x' & d &c\end{pmatrix},
\begin{pmatrix}\lambda & \mu & \gamma & \delta\cr \delta & \lambda & \mu & \gamma\cr \gamma & \delta &
\lambda & \mu\cr \mu & \gamma & \delta &\lambda\end{pmatrix},
$$}

\medskip

\noindent
where the parameters are all in ${\mathbb Z}_2$. Thus
$\vert X_{(12)}\vert=2^{10}, \vert X_{(123)}\vert=2^6, \vert X_{(12)(34)}\vert=2^8, \vert X_{(1234)}
\vert=2^4,$
and finally
$\Phi_4=\frac{1}{24}(2^{16}+6\cdot 2^{10}+8\cdot 2^6+3\cdot 2^8+6\cdot 2^4)=3044.$
\end{proof}

\medskip

The proposition above gives an idea of the super exponential growth of the number of non-isomorphic graphs of a
given order $n$ satisfying Condition (Sing). In this paper, we will deal only with the cases $n=1,2,3$ as only
those seems to be really tractable as far as atlases are concerned.

\medskip


\medskip

\section{Graphs of order one and two}

In this section we will classify the Leavitt path algebras of graphs with one and two vertices satisfying
Condition (Sing). The order one graphs satisfying Condition (Sing) offer no difficulty; they are collected in
the following table (it is well known that their associated Leavitt path algebras are $K$ and $K[x,x^{-1}]$).

\medskip

\begin{center}
\begin{tabular}{|c||c|}
\hline
$E$ &  $L_K(E)$     \\
\hline
&   \\
$\ \  \tiny\xymatrix{ {\bullet}  }   \quad \quad \quad \quad \I_1 \ \  $ & $K$ \\
&   \\
$\ \  \tiny\xymatrix{{\bullet} \ar@(dr,ur)}  \quad \quad  \quad \quad \I_2 \ \  $ & $K[x,x^{-1}]$  \\
&  \\
\hline
\end{tabular}
\smallskip

{\small Table 1: Case $n=1$.}
\end{center}

\medskip

The disconnected order two graphs satisfying Condition (Sing) are:

\medskip

$$\tiny
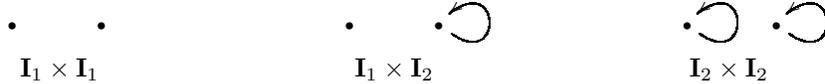
\xymatrix{ \bullet & \bullet &  &  & \bullet & \bullet \ar@(dr,ur) & & & \bullet \ar@(dr,ur) &\bullet
\ar@(dr,ur)}$$  \vskip-0.3cm  $$\xymatrix{ \ \ \ \ \ \ \ \ \  \I_1\times\I_1 &   &  & \ \ \I_1\times \I_2   &
& &\ \ \I_2\times\I_2 & }$$


The Leavitt path algebras associated to these three graphs are non-isomorphic since their socles ($K^2$, $K$
and $0$, respectively) are mutually non-isomorphic. Actually, $L_K(\I_1\times\I_1)\cong K\oplus K$,
$L_K(\I_1\times\I_2)\cong K\oplus K[x,x^{-1}]$ and $L_K(\I_2\times\I_2)\cong K[x,x^{-1}]\oplus K[x,x^{-1}]$.


Now we describe the Leavitt path algebras associated to order two connected graphs which satisfy Condition
(Sing). To this end we must study the orbits of the set $S$ of $2\times 2$ matrices with entries in ${\mathbb
Z}_2$ under the action of the group $S_2$ of row and column permutation (generated by the matrix
$\tiny\begin{pmatrix}0 & 1\cr 1 & 0\end{pmatrix}$). Thus, ruling out the matrices which stand for disconnected
graphs, the representatives of the orbits of $S$ are


$$\left\{\tiny\begin{pmatrix}1 & 1\cr 1 & 1\end{pmatrix}, \tiny\begin{pmatrix}1 & 1\cr 1 & 0\end{pmatrix},
\tiny\begin{pmatrix}1 & 0\cr 1 & 1\end{pmatrix}, \tiny\begin{pmatrix}1 & 1\cr 0 & 0\end{pmatrix},
 \tiny\begin{pmatrix}1 & 0\cr 1 & 0\end{pmatrix}, \tiny\begin{pmatrix}0 & 1\cr 1 & 0\end{pmatrix},
 \tiny\begin{pmatrix}0 & 1\cr 0 & 0\end{pmatrix}\right\}.$$


The seven matrices above do correspond to non-isomorphic graphs. However, some of them have isomorphic Leavitt
path algebras as can be shown by using a shift graph construction. For completeness we include here the basics
of this construction and refer the reader to \cite{AALP} for more information.

Let $E$ be a row-finite graph, and let $v, w\in E^0$ be distinct vertices which are not sinks. If there exists
an injective map $\theta: s^{-1}(w)\to s^{-1}(v)$ such that $r(e) = r(\theta(e))$ for every $e\in s^{-1}(w)$,
we define the \emph{shift graph from} $v$ \emph{to} $w$, denoted $F=E(w\hookrightarrow v)$, as follows:


\begin{enumerate}[{\rm (1)}]
\item $F^0 = E^0$.
\item $F^1 = (E^1 \setminus \theta (s^{-1}(w))) \cup \{f_{v,w}\}$, where $f_{v,w} \not\in E^1, s(f_{v,w}) =
v$ and $r(f_{v,w}) = w$.
\end{enumerate}


The key result about shift graphs is \cite[Theorem 3.11]{ALPS}, which states that for any row-finite graph
$E$, any shift graph $F=E(w\hookrightarrow v)$ produces a Leavitt path algebra isomorphic to $L_K(E)$. In what
follows we will analyze the relationship between the adjacency matrices $M$ and $N$ associated to the graphs
$E$ and $F$, respectively, when we assume that both graphs are finite, of the same order, and satisfy
Condition (Sing).


Thus, $M=(m_{kl})$ and $N=(n_{kl})$ are $n\times n$-matrices with entries in ${\mathbb Z}_2$. For fixed
$i,j\in\{1,\ldots,n\}$, we have $N=\Sh_{ij}(M)$ (equivalently $F=E(i\hookrightarrow j))$ when:
\begin{enumerate}[{\rm (1)}]
\item $m_{kl}= n_{kl}$ for all $k\neq j$ and all $l$.
\item $n_{jk}=m_{jk}-m_{ik}+\delta_{ki}$ for all $k$ (here $\delta$ is the Kronecker delta).
\end{enumerate}


In our case we find that $\tiny\begin{pmatrix}1 & 1\cr 1 & 0\end{pmatrix}=\Sh_{12}\tiny\begin{pmatrix}1
& 1\cr 1 & 1\end{pmatrix}$ and $\tiny\begin{pmatrix}0 & 1\cr 1 & 0\end{pmatrix}=\Sh_{21}\tiny\begin{pmatrix}1
& 0\cr 1 & 0\end{pmatrix}$. Also, no other shift process allows us to identify any other two matrices.
Hence, after collecting one representative of each orbit and applying the shift testing (see the Appendix for
the \emph{Magma} codes), we get the following set of matrices:


$$\left\{\tiny\begin{pmatrix}1 & 1\cr 1 & 1\end{pmatrix}, \begin{pmatrix}1 & 0\cr 1 & 1\end{pmatrix},
\begin{pmatrix}1 & 1\cr 0 & 0\end{pmatrix}, \begin{pmatrix}1 & 0\cr 1 & 0\end{pmatrix}, \begin{pmatrix}0 &
1\cr 0 & 0\end{pmatrix}\right\}.$$


These matrices correspond to the graphs we will denote $\mathbf{II}_1,\dots, \mathbf{II}_5$, which are given
by:


$$
\tiny\xymatrix{\bullet \ar@(dl,ul) \ar@/^.5pc/[r]  & \bullet \ar@(dr,ur)\ar@/^.5pc/[l]}
\hskip 1.6cm
\xymatrix{\bullet \ar@(dl,ul)   & \bullet\ar@(dr,ur)\ar[l]}
\hskip 1.6cm
\xymatrix{\bullet \ar@(dl,ul)\ar[r]& \bullet}
\hskip 1cm
\xymatrix{\bullet \ar@(dl,ul)  & \bullet \ar[l]}
\hskip 0.6cm
\xymatrix{\bullet \ar[r]& \bullet}
$$


With this last reduction, we have found a complete irredundant family of graphs of order two satisfying
Condition (Sing), i.e.,  whose Leavitt path algebras are non-isomorphic. In order to show this we will use
several invariants, namely, the ${\bf K}_0$ group, the socle, and the cardinal of the set of hereditary and
saturated subsets of vertices. We proceed to describe each of them.


Recall that a sink in $E$ is a vertex $i\in E^0$ such that $s^{-1}(i)=\emptyset$, that is, $i$ does not emit
any edge. The set of sinks of $E$ will be denoted by ${\rm Sink}(E)$. With this terminology we can summarize
the results on the ${\bf K}$-theory of the Leavitt algebra $L_K(E)$, obtained in \cite{ABC}, as follows.


Following \cite{AB} write $N_E$ and $1$ for the matrices in ${\mathbb Z}^{(E^0\times E^0\setminus
\text{Sink}(E))}$ obtained from $A^t_E$ and from the identity matrix after removing the columns corresponding
to sinks. Then there is a long exact sequence ($n\in {\mathbb Z}$)


$$\dots \to {\bf K}_n(K)^{(E^0\setminus {\rm Sink}(E))} \overset{1-N_E}{\longrightarrow} {\bf K}_n(K)^{(E^0)}
\longrightarrow {\bf K}_n(L_K(E)) \longrightarrow {\bf K}_{n-1}(K)^{(E^0\setminus {\rm Sink}(E))}.$$


In particular ${\bf K}_0(L_K(E)) \cong \text{coker}(1-N_E: {\mathbb Z}^{(E^0\setminus \text{Sink}(E))} \to
{\mathbb Z}^{(E_0)})$.


For a semiprime ring $R$, the \emph{socle} is the sum of all minimal left ideals of $R$ (equivalently, the
sum of all minimal right ideals of $R$) and is defined to be zero if there are no minimal one-sided ideals.


In order to compute the socle we need several results first. It has been proved in \cite[Theorem 4.2]{AMMS1}
that the socle of a Leavitt path algebra $L_K(E)$ is the ideal generated by the so called line points. We
recall the definitions here: a vertex $v$ in $E^0$ is a \emph{bifurcation} (or that \emph{there is a
bifurcation at} $v$) if $s^{-1}(v)$ has at least two elements. A vertex $u$ in $E^0$ will be called a
\emph{line point} if there are neither bifurcations nor cycles at any vertex $w\in T(u)$. We will denote by
$P_l(E)$ the set of all line points in $E^0$.


Our task here is to adapt  \cite[Theorem 4.2]{AMMS1} to our context, concretely we are interested in finding a
computational way to effectively compute the socle in the case of finite graphs. In this situation, each line
point connects to a sink, so that the ideal generated by all the line points connected to the same sink is
just the ideal generated by the sink. Thus the socle is the ideal generated by the sinks of the graph.


Hence we must compute the ideal of $L_K(E)$ generated by a sink $u$. Denoting such ideal by
$(u):=L_K(E)uL_K(E)$, it is clear (see \cite[Lemma 3.1]{AAPS}) that it is generated by the elements
$\mu\tau^*$ where $\mu,\tau$ are paths such that $r(\mu)=r(\tau)=u$ (either $\mu$ or $\tau$ can be the trivial
path $u$). To give an easier description of this ideal define $P_u$ as the set of all paths with range $u$.
Define also for each $\mu,\tau\in P_u$ the elements
$e_{\mu,\tau}:=\mu\tau^*,\quad e_{\mu}:=e_{\mu,\mu}=\mu\mu^*.$


All
are in $(u)$ and, moreover, it is easy to check that $\{e_\mu\}_{\mu\in P_u}$ is a
\emph{connected} set of pairwise orthogonal idempotents, i.e., $e_\mu L_K(E) e_\tau\ne 0$ for each
$\mu,\tau\in P_u$, because $0\ne\mu\tau^*=e_\mu(\mu\tau^*)e_\tau\in e_\mu L_K(E)e_\tau$. Another useful
property is given in the following lemma.


\begin{lemma}\label{connected} Let $E$ be a finite graph. For any two paths $\mu$ and $\tau$ such that
$r(\mu)$ and $r(\tau)$ are sinks we have:
$$e_\mu L_K(E) e_\tau =\begin{cases}
Ke_{\mu,\tau} & \text{ if   }\ r(\mu)=r(\tau) \cr
0 & \textrm{  otherwise.}\end{cases}$$
\end{lemma}
\begin{proof} Assume that both $\mu$ and $\tau$ are nontrivial paths. Consider a
generator $\omega:=\mu \mu^*(fg^*)\tau\tau^*$ of $e_{\mu}L_K(E)e_{\tau}$ where $f, g \in E^1$. If $\omega$ is
nonzero then $\mu=f\mu'$, where $r(\mu')=r(\mu)=:u$ (which is a sink), so
$\omega=\mu\mu'^*f^*fg^*\tau\tau^*=\mu\mu'^*g^*\tau\tau^*$. On the other hand, $\tau=g\tau'$ for some path
$\tau'$ such that $r(\tau')=r(\tau)=:v$ (again a sink). Consequently $\omega= \mu\mu'^*g^*g\tau'\tau^* =
\mu\mu'^*\tau'\tau^*$.


Continuing in this way, we can keep on canceling out edges of the paths $\mu'^*$ and $\tau'$. If they have
distinct length, say $l(\mu'^*)>l(\tau')$ then $\mu'=\tau'\mu''$, with $\mu''$ a nontrivial path. But this is
impossible because $s(\mu'')=r(\tau')=r(\tau)$ is a sink. Then $l(\mu'^*)=l(\tau')$ so that
$\omega=\mu\mu'^*\tau'\tau^*=\mu\tau^*=e_{\mu,\tau}$ as needed. Finally, with obvious modifications, we can
prove it when either $\mu$ or $\tau$ are vertices.
\end{proof}


Recall that an idempotent $e$ in a ring $R$ is said to be a \emph{division idempotent} if $eRe$ is a division
ring.


\begin{lemma} Let $u$ be a sink of a finite graph $E$. Then $\{e_\mu\}_{\mu\in P_u}$ is a set of pairwise
orthogonal and connected division idempotents.
\end{lemma}
\begin{proof} Suppose that the idempotents are not pairwise orthogonal. Then there exist two different paths
$\mu,\tau\in P_u$ such that $e_\mu e_\tau=\mu\mu^*\tau\tau^*\neq 0$. In this situation only two
things can happen: either $\tau=\mu \mu'$ for some path $\mu'$ or $\mu=\tau \tau'$ for some path $\tau'$.
Since $\mu\neq\tau$ by hypothesis, then $\mu'$ (respectively $\tau'$) is nontrivial, and this is
not possible since it must start at $s(\mu')=r(\mu),$ which is a sink (respectively, at $s(\tau')=r(\tau)$).


Any two idempotents $e_\mu$ and $e_\tau$ are connected by Lemma \ref{connected}, that is, $e_\mu L_K(E)
e_\tau=K e_{\mu,\tau}\ne 0$ and each $e_\mu$ is a division idempotent because $e_\mu L_K(E)e_\mu$ is
one-dimensional (apply Lemma \ref{connected}).
\end{proof}


Putting together all the information and the previous results above, we get the desired computer-friendly
description of the socle (see \cite{MM} for the implementations and explanations of the socle-related
\emph{Mathematica} code).


\begin{proposition} Let $E$ be a finite graph and $u_1,\ldots,u_n$ be the
sinks of $E$. Then $$\soc(L_K(E))\cong{\mathcal M}_{\vert P_{u_1}\vert}(K)\oplus\cdots\oplus
{\mathcal M}_{\vert
P_{u_n}\vert}(K),$$ where $\vert P_{u_i}\vert=\infty$ if $P_{u_i}$ contains paths with cycles.
\end{proposition}


The final result we will introduce in this section concerns the hereditary and saturated subsets of graphs
whose Leavitt path algebras are isomorphic.


\begin{proposition}\label{hs} Let $E$ and $F$ be row-finite graphs and let $\varphi: L_K(E)\to L_K(F)$ be a
ring isomorphism (not necessarily graded). Then:
\begin{enumerate}[{\rm (i)}]
\item If $I$ is a graded ideal of $L_K(E)$, then $\varphi(I)$ is a graded ideal of $L_K(F)$.
\item $|{\mathcal H}_E|=|{\mathcal H}_F|$.
\end{enumerate}
\end{proposition}
\begin{proof}
(i). An ideal $I$ in $L_K(E)$ is a graded ideal if and only if it is generated by idempotents; in fact
$I=I(H)$, where $H=I\cap E^0\in {\mathcal H}_E$ (see the proofs of \cite[Proposition 5.2 and Theorem
5.3]{AMP}). Since ring isomorphisms preserve idempotents, the ideal $\varphi(I)$ is generated by idempotents
too, and hence it is graded.


(ii). By \cite[Theorem 5.3]{AMP} there exists a lattice isomorphism between ${\mathcal H}_E$ and ${\mathcal
L}_{gr}(L_K(E))$ (the lattice of graded ideals of $L_K(E)$). Now (i) implies the result.
\end{proof}


\begin{definition}\label{HS}{\rm
We define ${\rm HS}_E$ (or HS when the graph is known) to be the number $|{\mathcal H}_E|-2$. By Proposition
\ref{hs}, it is an invariant for Leavitt path algebras.}
\end{definition}


The way to proceed in order to classify the Leavitt path algebras coming from order two graphs will be to
first arrange the Leavitt path algebras according to their ${\bf K}_0$ groups and socles. Only two graphs
agree on this data. For those, we compute $HS$ in order to distinguish their Leavitt path algebras. We collect
this information in Table 2.


Further, we have included an explicit algebraic description of $L_K(E)$ when this algebra is known; when it is not known
we have included the symbol ``$-$": the
eighth algebra is $L(1,2)$ as can be shown by doing an out-split to the rose of $2$-petals (see for instance
\cite[Definition 2.6 and Theorem 2.8]{AALP}); the  fifth algebra is the algebraic Toeplitz algebra ${\mathcal
T}$ (several representations of this algebra have been given: as an algebra defined in terms of generators
and relations in \cite{J}; via endomorphisms of an infinite dimensional vector space in \cite{Ger}, and as a
Leavitt path algebra in \cite{S}; actually an explicit isomorphism between the Leavitt path algebra
representation and the description given by Jacobson appears in \cite[Examples 4.3]{AB}); the isomorphism for
the fourth one can be found in \cite[Corollary 3.4]{AAS2}; the rest is folklore (see for example \cite{AA1}).

\bigskip

\begin{center}
\begin{tabular}{|c|c|c|c||c|}
\hline
$E$  & ${\bf K}_0$ & $\soc $ & HS & $L_K(E)$  \\
\hline
& & & & \\
$\tiny\xymatrix{ \bullet & \bullet}
$  & $\Z^2$ &  $K^2$ &  & $K^2$ \\
& & & & \\
$\tiny\xymatrix{\bullet \ar[r]& \bullet}
$ & $\Z$ & $\mathcal M_2(K)$  & &  $\mathcal M_2(K)$ \\
& & & & \\
$\tiny\xymatrix{\bullet \ar@(dl,ul) & \bullet }    $  & $\Z^2$ &  $K$  &  & $K \oplus K[x,x^{-1}]$ \\
& & & & \\
$\tiny\xymatrix{\bullet \ar@(dl,ul)  & \bullet \ar[l]}    $  & $\Z$ & $0$  & $0$  &
 $\mathcal M_2(K[x,x^{-1}])$\\
& & & & \\
$\tiny\xymatrix{\bullet \ar@(dl,ul)\ar[r]& \bullet}
$  & $\Z$ & $\mathcal M_\infty(K)$   &  & $\mathcal T$\\
& & & & \\
$\tiny\xymatrix{\bullet \ar@(dl,ul) &\bullet \ar@(dr,ur)}
$  & $\Z^2$ & $0$  &  & $K[x,x^{-1}]^2$ \\
& & & & \\
$\tiny\xymatrix{\bullet \ar@(dl,ul)   & \bullet\ar@(dr,ur)\ar[l]}
$ & $\Z$ &  $0$ & $1$ & ---  \\
& & & & \\
$\hskip 0.6cm\tiny\xymatrix{\bullet \ar@(dl,ul) \ar@/^.5pc/[r]  & \bullet \ar@(dr,ur)\ar@/^.5pc/[l]}
\hskip 0.6cm
$ & $0$ & $0$  & & $L(1,2)$ \\
& & & & \\
\hline
\end{tabular}

\smallskip

{\small Table 2: Case $n=2$.}
\end{center}


We collect all the information above in the next theorem.

\begin{theorem}\label{casetwo} There exist exactly $8$ mutually non-isomorphic Leavitt path algebras
in the family $\mathcal L_2=\{L_K(E)\ |\ E\, $
satisfies Condition {\rm (Sing)} and
$ |E^0|=2 $ and a set
of graphs whose Leavitt path algebras are those in $\mathcal L_2$ is given in Table $2$. A complete system of
invariants for $\mathcal L_2$ consists of the triple {\rm(${\bf K}_0$, $\soc$, {\rm HS})}. Concretely, two
Leavitt path algebras in $\mathcal L_2$, $L_K(E)$ and $L_K(F)$, are isomorphic as rings if and only if the
data of the previous invariants for $E$ and $F$ coincide.
\end{theorem}


\section{Graphs of order three}

Now we investigate the Leavitt path algebras associated to graphs of three vertices satisfying Condition
(Sing). Their adjacency matrices are the elements of $\mathcal M_3({\mathbb Z}_2)$. There are $2^9=512$ such
matrices but, as in the previous section, we must consider the orbits of this set under the action of de
subgroup of $\GL_3(\Z_2)$ generated by the matrices $I_{12}$, $I_{13}$, $I_{23}$. This subgroup is isomorphic
to $S_3$ and so it defines an action by conjugation on the set of binary matrices $\mathcal M_3({\mathbb
Z}_2)$. If we let the group $S_3$ act on the set of $512$ matrices we find the representatives of the orbits,
which form a set ${\mathcal P}$ of $104$ matrices, by Proposition \ref{paluego}.

\medskip

We explain below the procedure that has been used to generate the list containing the $104$ matrices
representing the graphs we are interested in (for the \emph{Magma} code see the Appendix). We create the
matrix algebra $\mathcal M_3({\mathbb Z}_2)$ of order three matrices over the field of two elements. Then $S3$
is the \emph{Magma} name for the symmetric group $S_3$ of permutations of three elements and $X$ is the
underlying set of $\mathcal M_3({\mathbb Z}_2)$.

\medskip

The function $p2m$ carries out the standard isomorphism which passes from a permutation of $S_3$ to a $3\times
3$ matrix as indicated at the beginning of Section $2$. The list $gen$ contains the generators of $S_3$ in
matrix form and then $S3m$ is the subgroup of $\GL_3(\Z_2)$ isomorphic to $S_3$. The function $f\colon X\times
S_3\to X$ gives the standard action of $S_3$ on $X$. Thus, we define $M$ as the $S_3$-set given by the action
$f$. Finally, $O$ is the set of orbits of $M$ under the action of $S_3$ and ``reducedlist" is $O$ transformed
in a list of elements.

\medskip

In the set ${\mathcal P}$ containing the representatives of the orbits of $\mathcal M_3({\mathbb Z}_2) $, we
define the relation $\sim$ such that: $m \sim n$ if and only if $n\equiv\Sh_{i,j}(m)$ or $
m\equiv\Sh_{i,j}(n)$ for some $i,j\in\{1,2,3\}$ (we use the notation $\equiv$ to indicate that the two
matrices are in the same orbit under the action of $S_3$).

\medskip

Thus, for each matrix in $p\in {\mathcal P}$, we compare it with all the other matrices $q\in {\mathcal P}$
and remove $q$ from ${\mathcal P}$ in case $p \sim q$. In this way we obtain a smaller set ${\mathcal Q}
\subseteq {\mathcal P}$ whose cardinal is $52$ and with the property that no two elements in ${\mathcal Q}$
are related via $\sim$. So the algebras that we must study are the Leavitt path algebras of the graphs
represented by these $52$ matrices.

\medskip

Our final task will be to find out all the graphs corresponding to non-isomorphic Leavitt path algebras that
arise from order 3 graphs. To this end, we arrange in different tables the Leavitt path algebras according to
their ${\bf K}_0$ groups and socles (if they are zero or not). Then, for each of these tables we compute, in a
systematic way, several invariants that will allow us to distinguish the Leavitt path algebras that are
different. For those which are indistinguishable, we actually provide ring isomorphisms between them.

\medskip

The tables are arranged as follows. In the first column we include the graphs that we have obtained after
choosing one representative of every orbit and after removing the shift graphs. The graphs have been ordered,
for an easier location, first by number of edges and then by number of disjoint cycles (that is, cycles which
do not share common edges).

\medskip

Only for the tables corresponding to nonzero socle do we include the computation of the socles and the
quotients $L_K(E)/Soc(L_K(E))$ (that we will denote by $Soc$ and $L/Soc$, respectively). The next columns will
contain, only when the information is both needed and useful (in the sense that they provide some
discrimination between at least two graphs), some other invariants that we proceed to describe here.

First we will compute the element $[1_{L_K(E)}]$ of ${\bf K}_0(L_K(E))$, which we know (see \cite{ALPS}) is
represented by $(1,1,\dots, 1)^t + \text{im}(I-N_E)$ in $\text{coker}(I-N_E)$.


The next invariant, provided by Corollary \ref{numberloops}, will allow us to discriminate the graphs that
contain a different number of isolated loops. The key point will be to give a ring-theoretic property for
Leavitt path algebras that contain isolated loops (Proposition \ref{inspired}), which can be regarded as an
analogue of a result that deals with graphs containing isolated vertices (result that was proved in
\cite[Proposition 2.3]{AAS1}). We include here an alternative proof using \cite[Proposition 3.1]{AMMS1}.


\begin{proposition}\label{onedimensional}
A Leavitt path algebra $L_K(E)$ contains a one-dimensional ideal (which is isomorphic to $K$) if and only if
$E$ contains an isolated vertex $u$. In this case $L_K(E)=Ku \oplus J$, where $J$ is an ideal isomorphic
to $L_K(F)$ and $F$ is the quotient graph $E/\{u\}$.
\end{proposition}
\begin{proof} Suppose that $I$ is a one-dimensional ideal of $L_K(E)$ and consider a nonzero element $x\in I$.
Applying \cite[Proposition 3.1]{AMMS1} we have two possibilities:


(i) There is a vertex $u\in I$. Then, $u$ is the unique vertex in $I$ because the dimension of $I$ is one.
Moreover, $I$ does not contain any edges whose range or source is $u$, because if $f$ is in this case, then
$f=fu\in I$ or $f=uf\in I$, which would imply that the dimension of $I$ is strictly bigger than one by
\cite[Lemma 1.1]{S}. Thus $u$ is an isolated vertex in $E^0$.


(ii) There is a cycle $c$ without exits based at a vertex $v$ and a nonzero polynomial $p:=p(c,c^*)\in I$. If
$p$ is a scalar multiple of $v$ we can argue as in case (i). So we may suppose $p\not\in Kv$. In this case it
is easy to prove that $\{p,p^2\}$ is a linearly independent subset of $I$, which is not possible by
hypothesis.


Hence, $I=K u$ for $u$ an isolated vertex and $H:=E^0\setminus \{ u \}\in \mathcal{H}_E$. Finally, the fact
that $L_K(E)= K u \oplus J$, where $J=I(H)$,  is straightforward.


The converse is trivial.
\end{proof}


\begin{proposition}\label{inspired}
A Leavitt path algebra $L_K(E)$ contains a graded ideal $I$ isomorphic to $K[x,x^{-1}]$ if and only if $E$
contains an isolated single loop graph based at a vertex $u$. In this case $I\cap E^0=\{u\}$ and
$L_K(E)=I\oplus J$ where $J$ is an ideal of $L_K(E)$ isomorphic to $L_K(F)$ where $F$ is the quotient graph
$E/\{u\}$.
\end{proposition}
\begin{proof} Suppose that $L_K(E)$ contains a graded ideal $I$ isomorphic to $K[x,x^{-1}]$. Then, by
\cite[Corollary 3.3 (1)]{AMMS2}, there is some $u\in I\cap E^0$. Since $I$ is a domain, it cannot contain
nontrivial orthogonal idempotents, so we have $I\cap E^0=\{u\}$.


Apply first \cite[Lemma 1.2]{AP} to get that $I\cong L_K(_HE)$, where $H=I\cap E^0$. It is clear that $u$ is
the only vertex contained in $I$ (as otherwise, $I$ would contain two orthogonal idempotents). Moreover, $u$
cannot be an isolated vertex in $E$ as otherwise, by Proposition \ref{onedimensional}, $I \cong Ku \oplus
L_K(G)$ (for a certain graph $G$). Since $I$ is a domain, then $L_K(G)=0$ and so $I\cong Ku \cong K\not\cong
K[x,x^{-1}]$.


Let $f$ be an edge in $E^1$ such that either $s(f)=u$ or $r(f)=u$. In both cases $f,f^*\in I$. Since $I$ is a
domain $ff^*=f^*f=r(f)\in I\cap E^0=\{u\}$, so that $r(f)=u$. Note that $ff^*=u$ also implies that $s(f)=u$,
and by relation (CK2), that $s^{-1}(u)=\{f\}$. Thus, $L_K(E)= I \oplus J$, for $J$ the graded ideal generated
by the hereditary and saturated set $E^0\setminus \{u\}$.


The converse is obvious.
\end{proof}


\begin{corollary}\label{invariantloop}
Let $E$ and $F$ be row-finite graphs such that $L_K(E)\cong L_K(F)$ as rings. Then $E$ has an isolated loop if
and only if so does $F$.
\end{corollary}
\begin{proof}
Consider $\varphi: L_K(E)\to L_K(F)$, a ring isomorphism and suppose that $E$ contains an isolated loop. By
Proposition \ref{inspired}, $L_K(E)$ contains a graded ideal $I$ isomorphic to $K[x,x^{-1}]$. By Proposition
\ref{hs} (i), $\varphi(I)$ is a graded ideal of $L_K(F)$. Since it is isomorphic to $K[x,x^{-1}]$, another
application of Proposition \ref{inspired} gives the result.
\end{proof}


\begin{corollary}\label{numberloops}Let $E$ and $F$ be row-finite graphs such that $L_K(E)\cong L_K(F)$ as
rings. Then $E$ has exactly $n$ different isolated loops if and only if so does $F$.
\end{corollary}
\begin{proof}
Denote by $n_E$ and $n_F$ the number of isolated loops in $E$ and $F$, respectively.


Let $f: L_K(E)\to L_K(F)$ be a ring isomorphism. If $n_E=0$, by Corollary \ref{invariantloop}, $n_F=0$. Let
$I$ be an ideal of $L_K(E)$ generated by an isolated loop based at a vertex $u\in E^0$. By Proposition
\ref{inspired}, $L_K(E)=I\oplus A$, where $A\cong L_K(E/\{u\})$. Denote by $J=f(I)$. As shown in the proof of
Proposition \ref{inspired}, $J$ is generated by an isolated loop based at a vertex $v\in F^0$ and $L_K(F)= J
\oplus B$, where $B\cong L_K(F/\{v\})$.


Then $A\cong B$ and we repeat the same reasoning taking into account that $n_E=1+n_{(E/\{u\})}$ and
$n_F=1+n_{(F/\{v\})}$. If either $n_E$ or $n_F$ is finite, then a descending process shows that $n_E=n_F$.
Otherwise both are countable and hence equal.
\end{proof}


\begin{definition}{\rm
We define ILN (isolated loops number) as the number of isolated loops in a row-finite graph $E$. By Corollary
\ref{numberloops}, this number is an invariant for Leavitt path algebras.}
\end{definition}


The following invariant we will consider in our classification task will be HS, already explained (see
Definition \ref{HS}), and in case $HS=1$ we use the following result.


\begin{proposition} Let $E$ and $F$ be row-finite graphs such there exists a ring isomorphism
$\varphi:L_K(E)\to L_K(F)$. Suppose that ${\rm HS}_E=1={\rm HS}_F$ and let $I$ and $J$ be the only nontrivial
graded ideals of $L_K(E)$ and $L_K(F)$, respectively. Then $J=\varphi(I)$ and $L_K(E)/I\cong L_K(F)/J$.
\end{proposition}
\begin{proof}
By Proposition \ref{hs} (1), $\varphi(I)$ is a graded ideal, and since $0\neq I\neq L_K(E)$ and ${\rm HS}_F=1$,
then $\varphi(I)=J$. Using this fact, the result follows.
\end{proof}

\vfill\eject

Thus, the proposition above shows that the quotient $L_K(E)/I(H)$, for the case that $I(H)$ is the only
nontrivial graded ideal, is an invariant that we will denote by L/I.


The final invariant that we will need is denoted by MT3+L, and it characterizes when a Leavitt path algebra is
primitive, as was proved in \cite[Theorem 4.6]{APS2}. Recall that a graph $E$ satisfies \emph{Condition} (MT3)
if for every $v,w\in E^0$ there exists $u\in E^0$ such that $v\geq u$ and $w\geq u$.


Note that this order of considering the invariants is consistent for all the cases $n=1,2,3$ because for the
two graphs that had to be distinguished in case $n=2$, namely the fourth and the seventh graph in Table 2,
they had both the same $[1_{L_K(E)}]$, and the same ILN, so they gave no information.

\medskip

Finally, in the last column of the tables, and as we did in the $n=2$ case, we have included an explicit
algebraic description of $L_K(E)$ when this algebra is known.

\medskip

\subsection{Nonzero socle and  ${\bf K}_0=\Z$}

In this situation, after taking one representative of every orbit and after eliminating shift graphs as we have
explained, the \emph{Magma} code gave an output of $9$ graphs. In the following table we show that all of them
actually provide non-isomorphic Leavitt path algebras and that, in our list of invariants, it is enough if we
stop at $[1_{L_K(E)}]$.

\medskip

The isomorphisms of the Leavitt path algebras of the first and second graphs can be obtained by
\cite[Proposition 3.5]{AAS1}. The Leavitt path algebra of the third graph, call it $E$, is the Toeplitz
algebra ${\mathcal T}$ as follows: first we observe that the unique possible out-split of the graph $\II_3$
gives

\medskip

\vskip 0.1cm
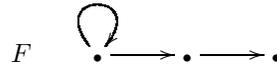
$$\begin{matrix} F &  & \tiny\xymatrix{  \bullet\ar@(ul,ur)\ar[r] & \bullet \ar[r] & \bullet }\end{matrix}$$

\medskip

\noindent which it turn gives the third graph of the previous table by a shift process. Hence by
\cite[Theorem 2.8]{AALP} and \cite[Theorem 3.11]{ALPS} we get that ${\mathcal T}\cong L_K(\II_3)\cong
L_K(F)\cong L_K(E)$.

\medskip

\begin{center}
\begin{tabular}{|c|c|c|c||c|}
\hline
$E$  & $\soc $ & $L/\soc$ & $[1]$ & $L_K(E)$ \\
\hline
& & & & \\
$\ \ \ \ \tiny\xymatrix{   \bullet \ar[r] & \bullet  & \bullet\ar[l]  }\ \ \ \  \  $   & $\mathcal M_3(K)$ &
  &  & $\mathcal M_3(K)$ \\
$\ \ \ \tiny\xymatrix{    & \bullet  &  \cr
             \bullet \ar[rr] \ar[ur] & &\bullet \ar[ul] }\ \ \  $   & $\xymatrix{ \\ \mathcal M_4(K)}$ & &  &
              $\xymatrix{ \\ \mathcal M_4(K)}$\\
& & & & \\
$\ \ \ \tiny\xymatrix{   \bullet\ar@(dl,ul)\ar[r] & \bullet  & \bullet \ar[l]  }\ \ \  $  &
$\mathcal M_\infty(K)$ &
$K[x,x^{-1}]$ &   & ${\mathcal T}$ \\
$\ \ \ \tiny\xymatrix{   & \bullet   & \cr
            \bullet\ar@(dl,ul)  \ar[ur] & &\bullet \ar[ll] \ar[ul] }\ \ \  $   &
            $\xymatrix{ \\ \mathcal M_\infty(K)}$ & $\xymatrix{ \\ \mathcal M_2(K[x,x^{-1}])}$   &   &
             $\xymatrix{ \\ \text{---}}$ \\
& & & & \\
$\ \ \ \tiny\xymatrix{  \bullet \ar@(dl,ul)  \ar@/^.5pc/[r]  & \bullet\ar@(dr,ur)\ar@/^.5pc/[l]& \bullet  }\
 \ \  $   & $ K$ & & & $K\oplus L(1,2)$ \\
$\ \ \  \tiny\xymatrix{   & \bullet   & \cr
            \bullet \ar@(dl,ul)  \ar@/^.5pc/[rr] \ar[ur] & &\bullet  \ar@/^.5pc/[ll] \ar[ul] }\ \ \  $ &
             $\xymatrix{ \\ \mathcal M_\infty(K)}$ & $\xymatrix{ \\ L(1,2)}$ &  $\xymatrix{ \\ 2}$
              & $\xymatrix{ \\ \text{---}}$ \\
& & & & \\
$\ \ \  \tiny\xymatrix{    & \bullet   & \cr
            \bullet \ar@(dl,ul)    \ar[ur] & &\bullet  \ar[ll] \ar[ul]  \ar@(dr,ur) }\ \ \  $  &
            $\xymatrix{ \\ \mathcal M_\infty(K)}$ & $\xymatrix{ \\ L_K(\II_2)}$ &   &
            $\xymatrix{ \\ \text{---}}$ \\
 & & & & \\
$\ \ \  \tiny\xymatrix{  \bullet \ar@(dl,ul) \ar@/^.5pc/[r]  & \bullet\ar@(ul,ur) \ar@/^.5pc/[l]\ar[r]&
 \bullet  }\ \ \  $   & $\mathcal M_\infty(K)$ & $L(1,2)$ & $0$ & --- \\
$\ \ \ \tiny\xymatrix{   & \bullet   & \cr
  \bullet \ar@(dl,ul)  \ar@/^.5pc/[rr] \ar[ur] & & \bullet  \ar@/^.5pc/[ll] \ar[ul] \ar@(dr,ur) }\ \ \  $   &
  $\xymatrix{ \\ \mathcal M_\infty(K)}$
 & $\xymatrix{ \\ L(1,2)}$ & $\xymatrix{ \\ 1}$  & $\xymatrix{ \\ \text{---}}$ \\
& & & & \\
\hline
\end{tabular}

{\small Table 3.1: Nonzero socle and  ${\bf K}_0=\Z$.}
\end{center}
\vskip 0.1cm

\vfill\eject

\subsection{Nonzero socle and ${\bf K}_0=\Z^2$}

For this class we get $11$ graphs; again all of them have non-isomorphic Leavitt path algebras. However, in
this case, it is enough to compute, in our ordered list of invariants, until ILN (note that the only two
graphs for which ILN is computed, cannot be distinguished by $[1]$, as it is $(1,1)$ in the two cases).

\bigskip
\bigskip

The isomorphisms here are based on previous cases (see Table 2) and on several well-known facts such as: the
decomposition of Leavitt path algebras of disconnected graphs as direct sums of the Leavitt path algebras of
the connected components; the description of Leavitt path algebras of finite and acyclic graphs which give
the finite-dimensional ones (see \cite[Proposition 3.5]{AAS1}); or, in more generality, the description of
the Leavitt path algebras satisfying Condition (NE) (i.e., such that no cycle in the graph has an exit),
which give the noetherian Leavitt path algebras \cite[Theorems 3.8 and 3.10]{AAS2} as those which are finite
direct sums of finite matrices over $K$ or $K[x,x^{-1}]$.

\bigskip
\bigskip

\begin{center}
\begin{tabular}{|c|c|c|c||c|}
\hline
$E$ & $\soc $ & $L/\soc$ & ILN  & $L_K(E)$\\
\hline
& & & & \\
 $\ \ \tiny\xymatrix{ \bullet & \bullet \ar[r] & \bullet  }$\ \ & $K\oplus \mathcal M_2(K)$ &  & &
 $K\oplus \mathcal M_2(K)$ \\
& & & & \\
$\ \ \tiny\xymatrix{ \bullet  & \bullet \ar[l]\ar[r] & \bullet }$ \ \ & $\mathcal M_2(K)^2$ &   &
& $\mathcal M_2(K)^2$ \\
& & & & \\
$ \ \ \tiny\xymatrix{ \bullet & \bullet \ar@(ul,ur) & \bullet \ar[l] }$ \ \ &$K$  & $\mathcal
M_2(K[x,x^{-1}])$ & &
$K\oplus \mathcal M_2(K[x,x^{-1}])$ \\
& & & &  \\
$\ \ \tiny\xymatrix{ \bullet & \bullet\ar@(ul,ur)\ar[r] & \bullet }$ \ \ & $K\oplus \mathcal M_\infty(K)$ &
  &   &
$K \oplus \mathcal T$ \\
& & & & \\
$\ \ \tiny\xymatrix{ \bullet\ar@(ul,ur) & \bullet  & \bullet\ar[l] }$ \ \ & $\mathcal M_2(K)$ & $K[x,x^{-1}]$
& &
$K[x,x^{-1}]\oplus \mathcal M_2(K)$  \\
& & & & \\
 $\ \ \tiny\xymatrix{ \bullet\ar@(ul,ur) & \bullet \ar[l]\ar[r] & \bullet  }$\ \ & $\mathcal M_2(K)$ &
 $\mathcal M_2(K[x,x^{-1}])$ &  & $\mathcal M_2(K)\oplus \mathcal M_2(K[x,x^{-1}])$ \\
& & & & \\
 $\ \ \tiny\xymatrix{ \bullet & \bullet \ar@(ul,ur)\ar[l] \ar[r]& \bullet }$ \ \ & $\mathcal M_\infty(K)^2$ &
   &  & ---  \\
& & & &  \\
$\ \ \tiny\xymatrix{ \bullet & \bullet\ar@(ul,ur)\ar[r] & \bullet \ar@(ul,ur)  }$ \ \ & $K$ &  $L_K(\II_2)$
 & & ---  \\
& & & & \\
$\ \ \tiny\xymatrix{ \bullet\ar@(ul,ur) & \bullet \ar@(ul,ur)\ar[r] & \bullet }$ \ \ &  $\mathcal M_\infty(K)$&
$K[x,x^{-1}]^2$   & 1 &  $K[x,x^{-1}]\oplus \mathcal T$ \\
& & & &  \\
$\ \ \tiny\xymatrix{ \bullet\ar@(ul,ur) & \bullet \ar@(ul,ur)\ar[l] \ar[r]& \bullet }$ \ \ &
 $\mathcal M_\infty(K) $ &
 $L_K(\II_2)$ &  & --- \\
& & & &  \\
$\ \ \tiny\xymatrix{ \bullet \ar@(ul,ur)\ar[r] & \bullet & \bullet\ar@(ul,ur)\ar[l] }$\ \ &
 $\mathcal M_\infty(K)$ &
$K[x,x^{-1}]^2$   & 0 & --- \\
& & & & \\

\hline
\end{tabular}
\medskip
\bigskip

{\small Table 3.2: Nonzero socle and ${\bf K}_0=\Z^2$.}
\end{center}

\bigskip
\bigskip

\subsection{Nonzero socle and ${\bf K}_0=\Z^3$}

In this case we find $3$ graphs and also $3$ different Leavitt path algebras. However, now the socle suffices
to distinguish any two of them.

\bigskip
\bigskip

\begin{center}
\begin{tabular}{|c|c||c|}
\hline
$E$ & $\soc$  & $L_K(E)$   \\
\hline
& &  \\
$\ \ \tiny\xymatrix{ \bullet & \bullet & \bullet}$\ \ & $K^3$   & $K^3$  \\
& &   \\
$\ \ \tiny\xymatrix{ \bullet & \bullet  & \bullet \ar@(ul,ur)}$ \ \ & $K^2$ & $K^2\oplus K[x,x^{-1}]$ \\
& &   \\
$\ \ \tiny\xymatrix{ \bullet & \bullet \ar@(ul,ur) & \bullet\ar@(ul,ur) }$ \ \ & $K$ & $K\oplus
K[x,x^{-1}]^2$\\
& & \\
\hline
\end{tabular}
\smallskip

{\small Table 3.3: Nonzero socle and  ${\bf K}_0=\Z^3$.}
\end{center}

\vfill\eject

\subsection{Nonzero socle and ${\bf K}_0=\Z\times \Z_2$} We find only $2$ graphs which again give $2$ Leavitt
path algebras that are not isomorphic. In this case the socle gives no information (both have socle equal to
${\mathcal M}_\infty(K)$), but the quotient module the socle is enough to get this conclusion.

\begin{center}
\begin{tabular}{|c|c||c|}
\hline
$E$ & $L/\soc$ & $L_K(E)$     \\
\hline
& &  \\
$\ \ \tiny\xymatrix{   & \bullet   & \cr
            \bullet    \ar@/^.5pc/[rr] \ar[ur] & &\bullet  \ar@/^.5pc/[ll] \ar[ul] }$
            \ \  & $   \xymatrix{ \\ \mathcal M_2(K[x,x^{-1}])}$ & $\xymatrix{ \\ \mbox{---}}$ \\
&  & \\
$\ \ \ \ \tiny\xymatrix{   & \bullet   & \cr
            \bullet \ar@(dl,ul)  \ar[rr] \ar[ur] & &\bullet  \ar[ul] }$
             \ \  & $\xymatrix{ \\ K[x,x^{-1}]}$ & $\xymatrix{ \\ \mbox{---}}$ \\
& & \\
\hline
\end{tabular}

\smallskip

{\small Table 3.4: Nonzero socle and  ${\bf K}_0=\Z\times\Z_2$.}
\end{center}

\subsection{Zero socle and ${\bf K}_0=0$} This is a particular case, as we do obtain $3$ different graphs but
their Leavitt path algebras are isomorphic (hence they all have the same invariants so that we do not include
any on Table 3.5).

\medskip

\begin{center}
\begin{tabular}{|c||c|}
\hline
$E$ & $L_K(E)$ \\
\hline
&  \\
\ \ \ \ \

\tiny\xymatrix{  &\bullet \ar@/^.5pc/[dl] \ar[dr] &  \cr
               \bullet \ar@(dl,ul) \ar@/^.5pc/[ur] & & \bullet \ar[ll] \ar@(dr,ur) }

\ \ \ \ \ \ \ \ \ \ \ \ \

\tiny \xymatrix{ &  \bullet \ar[dl] \ar[dr]   & \cr  \bullet \ar@(dl,ul) \ar@/^.5pc/[rr] & &
\bullet\ar@(dr,ur)\ar@/^.5pc/[ll] }

\ \ \ \ \ \ \ \ \ \ \ \ \

\tiny\xymatrix{  & \bullet \ar@(ul,ur) \ar@/^.5pc/[dl] \ar[dr] &  \cr
              \bullet \ar@(dl,ul) \ar@/^.5pc/[ur] \ar@/^.5pc/[rr] & &\bullet \ar@/^.5pc/[ll]
               \ar@(dr,ur) } \ \ \ \ \ \ \ & $\xymatrix{ \\ L(1,2)}$ \\

& \\
\hline
\end{tabular}

\smallskip

{\small Table 3.5: Zero socle and  ${\bf K}_0=0$.}
\end{center}

The Leavitt path algebras of these graphs are purely infinite simple and have the same $[1_{L_K(E)}]$
(equal to $0$). Hence \cite[Proposition 4.2]{AALP} gives that they are all isomorphic to $L(1,2)$. It is
interesting that, at least for the case $n=3$, only in this table do we get graphs which give isomorphic
Leavitt path algebras, and this happens precisely when the algebras are purely infinite simple, so that we can
make use of the aforementioned Classification Question for purely infinite simple unital Leavitt path
algebras.

\subsection{Zero socle and ${\bf K}_0=\Z$} Our simplification process shows that there are $11$ different
graphs in this class. Here, and in the remaining tables, we have zero socle so that clearly the columns for
the socle and the quotient module the socle are useless, hence we must rely on the other invariants. Actually,
here we need to use all of them in order to see that the Leavitt path algebras of these graphs are all
non-isomorphic.

\medskip

The explicit isomorphisms can be obtained by previous cases (see Table 2), by decomposition into direct sums
as mentioned before and by applications of \cite[Theorem 3.8]{AAS2}. Hence, the table of the $11$ cases with
their corresponding set of date for the invariants is as follows.


\begin{center}
\begin{tabular}{|c|c|c|c|c|c||c|}
\hline
$E$  & $[1]$ & ILN & HS  & L/I &  MT3+L & $L_K(E)$  \\
\hline
&  & & & & & \\
$\ \ \tiny\xymatrix{ \bullet \ar[r] & \bullet \ar@(ul,ur) & \bullet \ar[l] }\ \ $
 & 3 &  &   &   &  & $\mathcal M_3 (K[x,x^{-1}])$\\
\
$\ \tiny\xymatrix{   & \bullet \ar[dl] \ar[dr] & \cr
            \bullet\ar@(dl,ul)   & &\bullet \ar[ll] }\ $  & $\xymatrix{ \\ 4}$
               &   & & &  & $\xymatrix{ \\ \mathcal M_4(K[x,x^{-1}])}$\\
& & & & & & \\
$\ \tiny\xymatrix{ \bullet \ar@(dl,ul) \ar[r] & \bullet \ar@(ul,ur) & \bullet \ar[l]}\ $  & $1$ & $0$  & $1$
& $K[x,x^{-1}]$ &  F & --- \\

 $\ \tiny\xymatrix{  & \bullet \ar[dl] \ar[dr] & \cr
            \bullet\ar@(dl,ul)   & &\bullet \ar@(dr,ur)\ar[ll] }\ $ & $\xymatrix{ \\ 2}$ & $\xymatrix{ \\ 0}$
             & $\xymatrix{ \\ 1}$ &  $\xymatrix{ \\ \mathcal M_2(K[x,x^{-1}])}$& & $\xymatrix{ \\ \text{---}}$
              \\
& & & & & & \\

$\ \tiny\xymatrix{  \bullet \ar@(dl,ul) \ar@/^.5pc/[r]  & \bullet\ar@(dr,ur)\ar@/^.5pc/[l] & \bullet
 \ar@(dr,ur) }\ $
   & $1$ & $1$  &     & & & $L(1,2) \oplus K[x,x^{-1}]$  \\
& & & & & & \\

$\ \tiny\xymatrix{   & \bullet \ar[dl] \ar@/^.5pc/[dr] & \cr
            \bullet\ar@(dl,ul)   & &\bullet \ar@(dr,ur) \ar[ll] \ar@/^.5pc/[ul]}\ $  &
            $\xymatrix{ \\ 2}$ & $\xymatrix{ \\  0}$ & $\xymatrix{ \\ 1}$ &
              $\xymatrix{ \\ L(1,2)}$ & & $\xymatrix{ \\ \mbox{---}}$ \\

& & & & & & \\
 $\ \tiny\xymatrix{   & \bullet\ar@(ul,ur) \ar[dl] \ar[dr] & \cr
            \bullet\ar@(dl,ul)   & &\bullet \ar@(dr,ur)\ar[ll] }\ $  &
            $\xymatrix{ \\ 1}$ & $\xymatrix{ \\ 0}$ & $\xymatrix{ \\ 2}$  &  &  & $\xymatrix{ \\ \mbox{---}}$ \\
& & & & & & \\

$\ \tiny\xymatrix{  \bullet \ar@(dl,ul) \ar@/^.5pc/[r]  & \bullet\ar[r]\ar@(ul,ur)\ar@/^.5pc/[l] & \bullet
\ar@(dr,ur) }\ $ & $0$ & $0$ & $1$    &  &   & --- \\
& & & & & & \\

& & & & & & \\
$\ \tiny\xymatrix{   & \bullet\ar@(ul,ur) \ar[dl] \ar[dr] & \cr
            \bullet\ar@(dl,ul)  \ar@/^.5pc/[rr] & &\bullet \ar@(dr,ur)\ar@/^.5pc/[ll] }\ $  &
            $\xymatrix{ \\ 1}$ & $\xymatrix{ \\ 0}$ & $\xymatrix{ \\ 1}$
             &  $\xymatrix{ \\ K[x,x^{-1}]}$  & $\xymatrix{ \\ \mbox{T}}$ & $\xymatrix{ \\ \mbox{---}}$ \\
& & & & & & \\
& & & & & & \\
$\ \ \ \ \ \tiny\xymatrix{   & \bullet\ar@(ul,ur) \ar[dl] \ar@/^.5pc/[dr] & \cr
            \bullet\ar@(dl,ul)   & &\bullet \ar@(dr,ur)\ar[ll] \ar@/^.5pc/[ul]}\ \ \ \ \ $  &
            $\xymatrix{ \\ 1}$ & $\xymatrix{ \\ 0}$ & $\xymatrix{ \\ 1}$
             &   $\xymatrix{ \\ L(1,2)}$ &  & $\xymatrix{ \\ \mbox{---}}$ \\
& & & & & & \\

$\ \tiny\xymatrix{ \bullet \ar@(dl,ul) \ar@/^.5pc/[r]  & \bullet\ar@(ul,ur)\ar@/^.5pc/[l] \ar@/^.5pc/[r]&
\bullet \ar@(dr,ur)\ar@/^.5pc/[l] }\ $  & $0$ & $0$ & $0$    &  & & --- \\
& & & & & & \\
\hline
\end{tabular}


{\small Table 3.6: Zero socle and  ${\bf K}_0=\Z$.}
\end{center}


\subsection{Zero socle and ${\bf K}_0=\Z^2$} In this situation we get $5$ graphs, once more providing $5$
different isomorphism classes of Leavitt path algebras. In order to prove this, two
invariants ($[1]$ and ILN) are sufficient.

\begin{center}
\begin{tabular}{|c|c|c||c|}
\hline
$E$ & $[1]$ & ILN & $L_K(E)$ \\
\hline
&  & & \\
&  & & \\
$\ \tiny\xymatrix{ \bullet \ar@(ul,ur) & \bullet \ar@(ul,ur) & \bullet \ar[l] }\ $ & $(2, 1)$ &    &
$K[x,x^{-1}]\oplus \mathcal M_2(K[x,x^{-1}])$ \\
& & & \\
$\ \tiny\xymatrix{ \bullet \ar@(ul,ur) & \bullet \ar[l] \ar[r]& \bullet \ar@(ul,ur)}\ $ & $(2, 2)$ &   &
 $\mathcal M_2(K[x,x^{-1}])^2$ \\
& & & \\
$\ \tiny\xymatrix{  \bullet \ar@(ul,ur) & \bullet \ar@(ul,ur) & \bullet \ar[l]\ar@(ul,ur) }\ $ & $(1,1)$ & $1$
& ---  \\
& & & \\
$\ \tiny\xymatrix{ \bullet \ar@(ul,ur) \ar[r]& \bullet \ar@(ul,ur) & \bullet\ar[l] \ar@(ul,ur)}\ $ & $( 1,1)$&
 $0$  & --- \\
& & & \\
$\ \tiny\xymatrix{ \bullet \ar@(ul,ur) &  \bullet \ar@(ul,ur)\ar[l] \ar[r]& \bullet \ar@(ul,ur)}\ $ & $( 1,0)$
 &  & ---  \\
& & & \\
\hline
\end{tabular}

\smallskip

{\small Table 3.7: Zero socle and  ${\bf K}_0=\Z^2$.}
\end{center}

\subsection{Zero socle and ${\bf K}_0=\Z^3$}

There is nothing to do in this case as we in fact obtain only one graph whose explicit isomorphism of
its Leavitt path algebra is clear.

\begin{center}
\begin{tabular}{|c||c|}
\hline
$E$ & $L_K(E)$ \\
\hline
&  \\
&  \\
$\ \tiny\xymatrix{ \bullet \ar@(ul,ur) & \bullet \ar@(ul,ur) & \bullet \ar@(ul,ur)}\ $   &
$K[x,x^{-1}]^3$ \\
& \\
\hline
\end{tabular}

\smallskip

{\small Table 3.8: Zero socle and  ${\bf K}_0=\Z^3$.}
\end{center}


\subsection{Zero socle and ${\bf K}_0=\Z_2$}

There are two graphs whose Leavitt path algebras are in the previous conditions, and their Leavitt path
algebras can be distinguished just by $[1_{L_K(E)}]$.

\begin{center}
\begin{tabular}{|c|c||c|}
\hline
$E$ & $[1]$ & $L_K(E)$ \\
\hline

$\ \ \ \ \tiny\xymatrix{   & \bullet  \ar[dl] \ar@/^.5pc/[dr] & \cr
            \bullet\ar@(dl,ul)  \ar@/^.5pc/[rr]  & &\bullet \ar@(dr,ur)  \ar@/^.5pc/[ll]\ar@/^.5pc/[ul]}\ \ \
             \ \ $ & $\xymatrix{ \\ \overline{0}}$ & $\xymatrix{ \\\mathcal M_2(L(1,3))}$ \\
&  & \\
&  & \\
$\ \ \ \tiny\xymatrix{ & \bullet\ar@(ul,ur) \ar@/^.5pc/[dl] \ar@/^.5pc/[dr] & \cr
            \bullet\ar@(dl,ul)\ar@/^.5pc/[ur] \ar@/^.5pc/[rr]  & &\bullet \ar@(dr,ur)\ar@/^.5pc/[ll]
             \ar@/^.5pc/[ul]}\ \ \ \ \ $ & $\xymatrix{ \\ \overline{1}}$ & $\xymatrix{ \\ L(1,3)}$ \\
& &  \\
\hline
\end{tabular}


{\small Table 3.9: Zero socle and  ${\bf K}_0=\Z_2$.}
\end{center}

The Leavitt path algebra of the first graph, denote it by $E$, has the same ${\bf K}_0$, $[1]$ and $\det (I-N_E)$ as the graph $F$ given by

$$
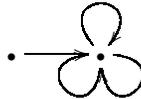
\tiny\xymatrix{
{\bullet}\ar[r] &{\bullet} \ar@(ul,ur) \ar@(r,d) \ar@(d,l)
}
$$

\medskip

whose Leavitt path algebra is isomorphic to ${\mathcal M}_2(L(1, 3))$. By \cite[Corollary 2.7]{ALPS}, both are isomorphic.

As far as the second graph is concerned, it is precisely the maximal out-split of the graph of the rose of $3$-petals
given by

$$\tiny
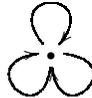
\xymatrix{ {\bullet} \ar@(ul,ur) \ar@(r,d) \ar@(d,l)}$$

\medskip

\noindent
and hence by \cite[Theorem 2.8]{AALP} its Leavitt path algebra is isomorphic to the classical Leavitt algebra
of type $(1,3)$, namely, $L(1,3)$.

\subsection{Zero socle and ${\bf K}_0=\Z\times \Z_2$} Only $2$ appear here, and they have non-isomorphic
Leavitt path algebras, as $[1_{L_K(E)}]$ shows.

\begin{center}
\begin{tabular}{|c|c||c|}
\hline
$E$ & $[1]$ & $L_K(E)$ \\
\hline
& & \\
$\ \ \ \ \tiny\xymatrix{   & \bullet \ar[dl] \ar@/^.5pc/[dr] & \cr
           \bullet\ar@(dl,ul)   & &\bullet \ar[ll] \ar@/^.5pc/[ul]}\ \ \ \ $ & $(2,\bar 0)$ & ---\\
& & \\
$\ \ \ \ \tiny\xymatrix{   & \bullet  \ar[dl]  & \cr
            \bullet\ar@(dl,ul)   & &\bullet  \ar[ul] \ar@(dr,ur) \ar[ll] }\ \ \ \ $ & $(1,\bar 0)$ & --- \\
& & \\
\hline
\end{tabular}

\smallskip

{\small Table 3.10: Zero socle and  ${\bf K}_0=\Z\times\Z_2$.}
\end{center}

\subsection{Zero socle and ${\bf K}_0=\Z_2^2$} For the remaining three cases, there is only one graph, so
that there is a unique Leavitt path algebra in each of these families too.


\begin{center}
\begin{tabular}{|c||c|}
\hline
$E$ & $L_K(E)$ \\
\hline
&  \\
$\ \tiny\xymatrix{  & \bullet \ar@/^.5pc/[dl] \ar@/^.5pc/[dr] & \cr
            \bullet\ar@/^.5pc/[ur] \ar@/^.5pc/[rr]  & &\bullet \ar@/^.5pc/[ll] \ar@/^.5pc/[ul]}\ $   &
--- \\
& \\
\hline
\end{tabular}

\smallskip

{\small Table 3.11: Zero socle and  ${\bf K}_0=\Z_2^2$.}
\end{center}


\subsection{Zero socle and ${\bf K}_0=\Z_3$} As mentioned, there is only one graph and therefore only one
Leavitt path algebra in this case.

\begin{center}
\begin{tabular}{|c||c|}
\hline
$E$ & $L_K(E)$ \\
\hline
&  \\
$\ \ \ \ \tiny\xymatrix{   & \bullet \ar@/^.5pc/[dl] \ar@/^.5pc/[dr] & \cr
            \ar@(dl,ul) \bullet\ar@/^.5pc/[ur] \ar@/^.5pc/[rr]  & &\bullet \ar@(dr,ur)\ar@/^.5pc/[ll]
             \ar@/^.5pc/[ul]}\ \ \ \ $   &
{L(1, 4)} \\
& \\
\hline
\end{tabular}
\smallskip

{\small Table 3.12: Zero socle and  ${\bf K}_0=\Z_3$.}
\end{center}

The Leavitt path algebra of the graph in the table has the same ${\bf K}_0$, $[1]$ and $det (I-N_E)$ as the graph of the 4-petals rose given by

$$
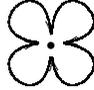
\tiny\xymatrix{
{\bullet} \ar@(u,r) \ar@(r,d) \ar@(d,l)\ar@(l, u)
}
$$
\smallskip

\noindent
whose Leavitt path algebra is isomorphic to $L(1, 4)$. By \cite[Corollary 2.7]{ALPS}, both are isomorphic.

\subsection{Zero socle and ${\bf K}_0=\Z_4$} The only graph here is given in the following table.

\begin{center}
\begin{tabular}{|c||c|}
\hline
$E$ & $L_K(E)$ \\
\hline
&  \\
$\ \ \ \ \tiny\xymatrix{   & \bullet \ar@/^.5pc/[dl] \ar@/^.5pc/[dr] & \cr
            \bullet \ar@(dl,ul) \ar@/^.5pc/[ur] \ar@/^.5pc/[rr]  & &\bullet
             \ar@/^.5pc/[ll] \ar@/^.5pc/[ul]}\ \ \ \ $   &
$ \mathcal{M}_2(L(1, 5)) $\\
& \\
\hline
\end{tabular}

\smallskip

{\small Table 3.13: Zero socle and  ${\bf K}_0=\Z_4$.}
\end{center}

The Leavitt path algebra of this graph has the same ${\bf K}_0$, $[1]$ and $det (I-N_E)$ as the graph given by

\bigskip

$$
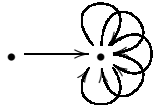
\tiny\xymatrix{
{\bullet}\ar[r] &{\bullet} \ar@(ul,ur) \ar@(u,r) \ar@(ur,dr)\ar@(r,d)\ar@(dr, dl)
}
$$

\bigskip
\medskip

\noindent
whose Leavitt path algebra is isomorphic to ${\mathcal M}_2(L(1, 5))$. By \cite[Corollary 2.7]{ALPS}, both are isomorphic.
\smallskip

We are finally in a position to precisely state the Classification Theorem for Leavitt path algebras of graphs
of order three that satisfy Condition (Sing), which summarizes the results that we have been obtaining
throughout this section.

\begin{theorem}\label{casethree} There exist exactly $50$ mutually non-isomorphic Leavitt path algebras
in the family $\mathcal L_3=\{L_K(E)\ |\ E $ satisfies Condition {\rm (Sing)} and $|E^0|= 3\}$ and a set of
graphs whose Leavitt path algebras are those in $\mathcal L_3$ is given in Tables $3.1,...,3.13$. A complete
system of invariants for $\mathcal L_3$ consists of the set {\rm{(${\bf K}_0$, $\soc$, $L/\soc$, $[1]$, ILN,
HS, $L/I$, MT3+L)}}. Concretely, two Leavitt path algebras in $\mathcal L_3$, $L_K(E)$ and $L_K(F)$, are
isomorphic as rings if and only if the data of the previous invariants for $E$ and $F$ coincide.
\end{theorem}

Our final result puts together all the cases $n=1,2,3$ so that we give a Classification Theorem for Leavitt
path algebras of graphs of order less than three that satisfy Condition (Sing), thus collecting all the
results, information and data that we have been developing throughout the paper.

\begin{theorem}\label{caselessthanthree} There exist exactly $57$ mutually non-isomorphic Leavitt path
algebras in the family $\mathcal L_{\leq 3}=\{L_K(E)\ |\ E \text{ satisfies Condition {\rm (Sing)} and }
|E^0|\leq 3 \}$ and a set of graphs whose Leavitt path algebras are those in $\mathcal L_{\leq 3}$ is given in
Tables 1,2,3.1,...,3.13. A complete system of invariants for $\mathcal L_{\leq 3}$ consists of the set
{\rm{(${\bf K}_0$, $\soc$, $L/\soc$, $[1]$, ILN, HS, $L/I$, MT3+L)}}. Concretely, two Leavitt path algebras in
$\mathcal L_{\leq 3}$, $L_K(E)$ and $L_K(F)$, are isomorphic as rings if and only if the data of the previous
invariants for $E$ and $F$ coincide.
\end{theorem}
\begin{proof}
It only remains to compare the different cases $n=1,2,3$ all at once. In order to do that, we will pick each
of the $10$ graphs of cases $n=1,2$ and, after computing the pair (${\bf K}_0$, $\soc$) we compare the rest of
the invariants. Concretely, for the graph $\I_1$ we have ${\bf K}_0(L_K(\I_1))=\Z$ and $\soc(L_K(\I_1))=K$.
The only graph with this data is the fifth graph in Table 3.1, call it $E$. However, we get that
$L_K(\I_1)/\soc(L_K(\I_1))=0\not\cong L(1,2)=L_K(E)/\soc(L_K(E))$.

For $\I_2$ we have $({\bf K}_0(L_K(\I_2)), \soc(L_K(\I_2)))=(\Z,0)$. Again, there is only one other graph with
this data, namely, the third one in Table 3.6. Applying our list of invariants, we first compute
$L_K(\I_2)/\soc(L_K(\I_2))=K[x,x^{-1}]$. Applying Proposition \ref{invariantloop} and Corollary
\ref{numberloops} we get that the fifth one, call it $F$, is the only possible graph in Table 3.6 whose
Leavitt path algebra could be isomorphic to $L_K(\I_2)$, but this does not happen as clearly
$L_K(\I_2)\not\cong L_K(F)$.

Let us focus on the case $n=2$. Unlike the previous case, now three graphs in Table 2 will give us Leavitt
path algebras which are isomorphic to some of case $n=3$, whereas the other five will produce non-isomorphic
Leavitt path algebras when compared to that of $n=3$, as we will show now.

The pairs $({\bf K}_0, \soc)$ for the first two graphs in Table 2 are different to any other such pair in the
other tables, so their Leavitt path algebras are not isomorphic to anyone appearing in the case $n=3$.

The Leavitt path algebra of the third graph in Table 2 has the same $({\bf K}_0, \soc)$ as the Leavitt path
algebras of the third and eighth graphs in Table 3.2, but when we compute $L/\soc$ we get three non-isomorphic
rings: $K[x,x^{-1}], \mathcal M_2(K[x,x^{-1}])$ and $L_K(\II_2)$.

For the fourth graph in Table 2 we have that the pair $({\bf K}_0, \soc)$ of its associated Leavitt path
algebra is $(\Z, 0)$, which could provide a Leavitt path algebra isomorphic to the Leavitt path algebra of
some graph in Table 3.6. As the quotients by their socles (we are considering the graphs in Table 3.6) give us
no known information, we jump on to the following invariant, namely, $[1_{L_K(E)}]$ which is $2$ in this case.
In this situation we have two graphs in Table 3.6, namely the fourth and sixth ones. We go on comparing
invariants and the three graphs have ${\rm ILN}=0$, but ${\rm HS}=0$ in our original graph while ${\rm HS}=1$
for the other two.

The Leavitt path algebra of the fifth graph is the Toeplitz algebra ${\mathcal T}$ which appears already in
Table 3.1.

For the sixth graph $\I_2^2$ we have to focus on Table 3.7. Since $[1_{L_K(\I_2^2)}]=(1,1)$, we compute ILN,
obtaining $2$ for $\I_2^2$ but $0$ or $1$ for all the graphs in Table 3.7.

The seventh graph in Table 2 gives a Leavitt path algebra isomorphic to that of the third graph in Table 3.6
as follows: by an out-split we obtain the graph

$$\tiny
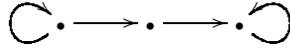
\xymatrix{ {\bullet} \ar@(dl,ul) \ar[r] & {\bullet} \ar[r] & {\bullet} \ar@(dr,ur)}$$

We note that this graph is the shift graph of the third graph in Table 3.6. Then apply \cite[Theorem
2.8]{AALP} and \cite[Theorem 3.11]{ALPS}.

Finally, the Leavitt path algebra of the last graph is $L(1,2)$ which also shows up in Table 3.5.

Hence, out of the $62$ graphs given in the tables we only obtain $2+(8-3)+(52-2)=57$ non-isomorphic Leavitt
path algebras.
\end{proof}

\begin{remark} {\rm A natural setting and way to use the previous theorem is this: we start with a
graph $E$ satisfying Condition (Sing) and such that $|E^0|\leq 3$ (note that this graph might not appear in
our tables). Thus Theorem \ref{caselessthanthree} guarantees that there is exactly one graph among the $57$
referred to in the statement, call if $F$, such that $L_K(E)\cong L_K(F)$ as rings. In order to find it, we
apply systematically the list of invariants to $E$ to narrow our search until we find $F$.}
\end{remark}

\begin{remark} {\rm As a corollary of our general Classification Theorem \ref{casethree}, we can obtain the
Classification Theorem for purely infinite simple unital Leavitt path algebras as stated in \cite[Proposition
4.2]{AALP}, by proceeding in some other fashion, as follows: among the $52$ graphs that we have obtained for
$n=3$, we single out those that provide purely infinite simple Leavitt path algebras. This task is
straightforward by using the graph-theoretic characterization of purely simple Leavitt path algebras as those
whose graph has $HS=0$, satisfy Condition (L) and every vertex connects to a cycle (see \cite[Theorem
11]{AA2}). One useful trick is the following: if a graph $E$ satisfies the three conditions above, then it
cannot contain a sink and it must be connected (these obvious observations actually rule out many graphs).

This leaves exactly $7$ graphs, namely: any of those appearing in Table 3.5 (the three have isomorphic Leavitt
path algebras), the last graph in Table 3.6, and all the graphs in tables 3.9, 3.11, 3.12 and 3.13. Finally
one checks that the data $({\bf K}_0(L_K(E)), [1_{L_K(E)}])$ is different for all these $7$ cases as is shown
in the tables.

We point out that just by looking at the tables one can clearly see that the information about ${\bf
K}_0(L_K(E))$ and $[1_{L_K(E)}]$ is not enough for classification of the Leavitt path algebras that are not
necessarily purely infinite simple.}
\end{remark}

\section{Appendix}

In this section we include the \emph{Magma} and \emph{Mathematica} codes needed for our computations. They
consist on a list of functions written in the order they have been used. The computation of the invariants has
been performed by the \emph{Mathematica} software. However, for the calculation of the orbits and shift graphs
the \emph{Magma} software has been used instead, as it has proved to be faster and more efficient for these
purposes.

\subsection{Magma codes}

We provide here a list of the routines that have been used together with a brief description of them.

\begin{itemize}
\item {\bf int}: given an $3\times 3$ matrix with entries in $\Z_2$, it returns the same
matrix considered as an element in $\mathcal M_3(\{0,1\})$.
\item {\bf zerorow}: given an integer $i$ and a matrix $m$, it returns TRUE if the $i$th row of $m$ is zero.
\item {\bf nonzerosoc}: given a matrix $m$ gives TRUE if $m$ has some zero row.
\item {\bf test}: given integers $i,j$ and a matrix $m$, it returns TRUE if the $i$th row is nonzero and each
element in the $i$th row is less or equal than the corresponding element in the $j$th row.
\item {\bf sing}: checks if the entries of a given matrix are all $\le 1$, i.e., verifies if Condition (Sing)
is satisfied.
\item {\bf sh}: let $m$ be the adjacency matrix of a direct graph of $n$ vertices and $i,j\in\{1,\ldots,n\}$.
Then $\hbox{sh}(i,j,m)$ performs the shift graph $\Sh_{i,j}(m)$. If the shift is not possible, the function
returns $m$.
\item {\bf ish}: given a matrix $m$, this function returns a matrix $x$ (if it exists) such that
$\Sh_{i,j}(x)=m$. If $x$ does not exist, then the function returns $m$.
\item {\bf ss}: given $m$, it returns a list containing all the matrices produced by a shift from $m$ and also
all those which give $m$ by applying a shift process to it.
\item {\bf comp}: given two matrices $x$ and $y$, it returns TRUE if there is a nonempty intersection between
$\hbox{ss}(y)$ and the orbit of $x$ (under the action of $S_3$) or between $\hbox{ss}(x)$ and the orbit of
$y$. Roughly speaking, this function returns TRUE if some shift or inverse shift of $x$ is in the same orbit
as $y$ or vice versa.
\item {\bf compressto}: given a matrix $x$ and a list, the function returns TRUE if $\hbox{comp}(x,y)$ is TRUE
for some $y$ in the list.
\end{itemize}

We include the \emph{Magma} code of all these functions.

\def\peq2{\fontsize{9}{9}\selectfont}

{\peq2

\begin{verbatim}
int:=function(x)
return MatrixAlgebra(IntegerRing(),n)!x;
end function;

zerorow:=function(i,m)
return (m[i,1] eq 0) and (m[i,2] eq 0) and (m[i,3] eq 0);
end function;

nonzerosoc:=function(m)
return zerorow(1,m) or zerorow(2,m) or zerorow(3,m);
end function;

test:=function(i,j,m)
local logical;
logical:=true;
for k:=1 to n do; logical:=logical and (int(m)[i,k] le ent(m)[j,k]); end for;
return (logical and not zerorow(i,m)); end function;

sing:=function(x)
local logical;
logical:=true;
for i:=1 to n do;
   for j:=1 to n do;
   logical:=logical and (x[i,j] le 1);
   end for;
end for;
return logical;
end function;

sh:=function(i,j,m)
local s;
s:=int(m);
if test(i,j,m) then
   for k:=1 to n do; s[j,k]:=s[j,k]-s[i,k]; end for;
s[j,i]:=s[j,i]+1; end if; if sing(s) then return s; else return m; end if;
end function;

ish:=function(i,j,m)
local s;
s:=int(m);
if s[j,i] eq 0 then return s;
   else s[j,i]:=s[j,i]-1;
   for k:=1 to n do;
   s[j,k]:=s[j,k]+s[i,k];
   end for;
end if;
if not zerorow(i,m) and sing(s) then return s; else return m; end if;
end function;

ss:=function(m)
local lista;
lista:={};
for i:=1 to n do;
   for j:=1 to n do;
   if not (i eq j) then Include(~lista,sh(i,j,m)); end if;
   end for;
end for;
for i:=1 to n do;
   for j:=1 to n do;
   if not (i eq j) then Include(~lista,ish(i,j,m)); end if;
   end for;
end for;
return lista;
end function;

comp:=function(x,y)
return (not (Orbit(S3,M,x) meet ss(y) eq {})) or
       (not(Orbit(S3,M,y) meet ss(x) eq {}));
end function;

compressto:=function(x,lista)
local logical,j;
logical:=false;
j:=1;
while (j le #lista) and not comp(x,lista[j]) do; j:=j+1; end while;
if j eq #lista+1 then return false; else return true; end if;
end function;

n:=3;
F:=FiniteField(2,1);
A:=MatrixAlgebra(F,n);
S3:=Sym(n);
X:=Set(A);
p2m:=function(p)
return PermutationMatrix(F,p);
end function;
gen:=[p2m(x): x in Generators(S3)];
S3m:=sub<GL_3(F)|gen>;
ptm:=hom<S3->S3m|x:->Transpose(PermutationMatrix(F,x))>;
f:=map<car<X,S3>->X|x:->ptm(x[2])*x[1]*ptm(x[2])^(-1)>;
M:=GSet(S3,X,f);
O:=Orbits(S3,M);
reducedlist:=[[x: x in O[i]][1]:i in [1..#O]];
reducedlist:=[int(x): x in reducedlist];
aux:=[];
while not (reducedlist eq []) do;
x:=reducedlist[1];Remove(~reducedlist,1);
if not compressto(x,reducedlist) then Include(~aux,x);
end if;
end while;
\end{verbatim}
}

\subsection{{\em Mathematica} implemented instructions}

Again, we provide first a list of the routines that have been used together with a brief description of them.

\begin{itemize}
\item {\bf Gr}: it represents the directed graph.
\item {\bf SinkQ}: checks if a vertex is a sink.
\item {\bf Redu}: diagonal form.
\item {\bf Pmatrix}: $P$-matrix associated to the previous diagonal form.
\item {\bf K}$_0$: computes the ${\bf K}_0$ group.
\item {\bf Unit}: computes the unit of the ${\bf K}_0$ group.
\item {\bf ConditionMT3Q}: checks the Condition (MT3).
\item {\bf ConditionLQ}: checks the Condition (L).
\item {\bf CofinalQ}: checks the cofinal condition.
\item Example: an example of how to construct classification tables.
\end{itemize}

Finally, we include the \emph{Mathematica} code of all these functions.

\medskip

\def\peq{\fontsize{9}{9}\selectfont}

\medskip

{\peq
\begin{tabular}{l}
${\rm {\bf Tograph}[m\_] := {\bf Module}[\{n, x\},}$\\
\hskip .5cm ${\rm n = {\bf Length}[m];}$\\
\hskip .5cm ${\rm x = {\bf Flatten}[{\bf Table}[i \to j, \{i, n\}, \{j, n\}]*m] // {\bf Union};}$\\
\hskip .5cm ${\rm {\bf If}[{\bf Length}[x[[1]]] == 0, {\bf Delete}[x, 1], x]]}$\\
\end{tabular}
}

\medskip

{\peq
\begin{tabular}{l}
${\rm {\bf Gr}[x\_] :=}$\\
\hskip .5cm ${\rm  {\bf GraphPlot}[{\bf Tograph}[x], {\bf DirectedEdges} \to {\bf True}, {\bf VertexLabeling}
\to {\bf True}]}$\\
\end{tabular}
}

\medskip

{\peq
\begin{tabular}{l}
${\rm {\bf SinkQ}[x\_, i\_] := {\bf If}[x[[i]] == 0 x[[i]], 0, 1];}$\\
\end{tabular}
}

\medskip

{\peq
\begin{tabular}{l}
${\rm {\bf << Algebra`IntegerSmithNormalForm`}}$\\
\end{tabular}
}

\medskip

{\peq
\begin{tabular}{l}
${\rm {\bf Redu}[x\_] :={\bf SmithForm}[}$\\
\hskip .5cm ${\rm {\bf Transpose}[x] - {\bf DiagonalMatrix}[{\bf Table}[{\bf SinkQ}[x, i], \{i,
{\bf Length}[x]\}]]];}$\\
\end{tabular}
}

\medskip

{\peq
\begin{tabular}{l}
${\rm {\bf Pmatrix}[x\_] := {\bf ExtendedSmithForm}[}$\\
\hskip .5cm ${\rm n = {\bf Transpose}[x] - {\bf DiagonalMatrix}[{\bf Table}[{\bf SinkQ}[x, i],
\{i,{\bf Length}[x]\}]]][[2, 1]]}$\\
\end{tabular}
}

\medskip

{\peq
\begin{tabular}{l}
${\rm Example\ of\ computing\ [1]}$\\
\hskip .5cm ${\rm {\bf Table}[\{list[[i]], {\bf Gr}[list[[i]]],
{\bf Redu}[list[[i]]],{\bf Pmatrix}[list[[i]]].}$ $\begin{pmatrix} 1\\ 1 \\ 1\end{pmatrix}$ ${\rm, \{i,
 {\bf Length}[list]\}]}$\\
\end{tabular}
}

\medskip

{\peq
\begin{tabular}{l}
${\rm {\bf Z}[x\_] := {\bf Which}[x == 0, Z, x == 1, 1, x > 1, Z_x];}$\\
\end{tabular}
}

\medskip

{\peq
\begin{tabular}{l}
${\rm {\bf K}_0[m\_] :={\bf Module}[\{x\}, x = {\bf Redu}[m];{\bf Product}[{\bf Z}[x[[i, i]]], \{i,
{\bf Length}[x]\}]]}$\\
\end{tabular}
}

\medskip

{\peq
\begin{tabular}{l}
${\rm {\bf myMod}[x\_,y\_]:={\bf If}[y\neq 0, {\bf Mod}[x,y],x]}$ \\
\end{tabular}
}

\medskip

{\peq
\begin{tabular}{l}
${\rm {\bf Unit}[x\_]:={\bf Module}[\{v,l\}, v={\bf Pmatrix}[x].\begin{pmatrix} 1\\ 1 \\
1\end{pmatrix}; l= {\bf Redu}[x];{\bf Table}[{\bf myMod}[v[[i]], l[[i,i]]], \{i,3\}]]}$ \\
\end{tabular}
}

\medskip

{\peq
\begin{tabular}{l}
${\rm {\bf NB}[m\_] := }$\\
\hskip .5cm ${\rm {\bf Module}[\{nm = m, l = {\bf Table}[0, \{k, {\bf Length}[m]\}], n = {\bf Length}[m], s,
k\},}$\\
\hskip .5cm ${\rm {\bf Do}[s = 0;}$\\
\hskip .5cm ${\rm {\bf Do}[s = s + m[[i, j]], \{j, n\}];}$\\
\hskip .5cm ${\rm {\bf If}[s > 1, l[[i]] = 1;}$\\
\hskip .5cm ${\rm {\bf Do}[nm[[i, k]] = 0; nm[[k, i]] = 0, \{k, n\}]], \{i, n\}];}$\\
\hskip .5cm ${\rm eli = {\bf Position}[l, 1]; k = 0;}$\\
\hskip .5cm ${\rm {\bf Do}[}$\\
\hskip .5cm ${\rm nm = {\bf Drop}[nm, eli[[i]] - k, eli[[i]] - k]; k++, \{i, {\bf Length}[eli]\}];}$\\
\hskip .5cm ${\rm  nm}$\\
\hskip .5cm ${\rm ]}$\\
\end{tabular}
}

\medskip

{\peq
\begin{tabular}{l}
${\rm {\bf << Combinatorica`}}$\\
\end{tabular}
}

\medskip

{\peq
\begin{tabular}{l}
${\rm {\bf ConditionLQ}[m\_] :=}$\\
\hskip .5cm ${\rm {\bf AcyclicQ}[{\bf FromAdjacencyMatrix}[{\bf NB}[m], {\bf Type} \to {\bf Directed}]]}$\\
\end{tabular}
}

\medskip

{\peq
\begin{tabular}{l}
${\rm {\bf lr}[li\_{\bf ?ListQ}, m\_] :=}$\\
\hskip .5cm ${\rm {\bf Union}[{\bf Flatten}[}$\\
\hskip .5cm ${\rm {\bf Map}[{\bf Cases}[m[[\#]]*{\bf Table}[j, \{j, {\bf Length}[m]\}], {\bf Except}[0]] \&,
li]]]
}$\\
\end{tabular}
}

\medskip

{\peq
\begin{tabular}{l}
${\rm {\bf Her}[li\_{\bf ?ListQ}, m\_] := {\bf Module}[\{H = li, G = {\bf Table}[k, \{k,
{\bf Length}[m]\}]\},}$\\
\hskip .5cm ${\rm {\bf While}[G != H, G = H; H = {\bf Union}[H, {\bf lr}[H, m]]]; H] }$\\
\end{tabular}
}

\medskip

{\peq
\begin{tabular}{l}
${\rm {\bf ConditionMT3Q}[m\_] :={\bf  Module}[\{n,l,re\}, n = {\bf Length}[m]; l = {\bf Table}[i,\{i, n\}];}$\\
${\rm   re = True;}$\\
${\rm  {\bf  Do}[{\bf If}[{\bf Intersection}[{\bf Her}[\{l[[i]]\}, m], {\bf Her}[\{l[[j]]\}, m]] == \{\},}$\\
${\rm  re = {\bf False}; {\bf Break}[]], \{i, n\}, \{j, n\}];}$\\
 re]
\end{tabular}
}

\medskip

{\peq
\begin{tabular}{l}
${\rm {\bf HSC}[li\_{\bf ?ListQ}, m\_] :=}$\\
\hskip .5cm ${\rm {\bf Module}[\{X, H, G, F, n, i\}, }$\\
\hskip .5cm ${\rm H = {\bf Her}[li, m]; G = {\bf Table}[k, \{k, {\bf Length}[m]\}]; F =
{\bf Complement}[G, H];}$\\
\hskip .5cm ${\rm n = {\bf Length}[F]; i = 1;}$\\
\hskip .5cm ${\rm {\bf While}[F != \{\} \&\& G != H \&\& i \leq n, X = {\bf lr}[\{F[[i]]\}, m];}$\\
\hskip .5cm ${\rm {\bf If}[X != \{\} \&\& {\bf Intersection}[X, H] == X, H = {\bf Union}[H, \{F[[i]]\}];}$\\
\hskip .5cm ${\rm F = {\bf Complement}[G, H]; n = {\bf Length}[F]; i = 1, i++]}$\\
\hskip .5cm ${\rm    ]; H]}$\\
\end{tabular}
}

\medskip

{\peq
\begin{tabular}{l}
${\rm ps[k\_] := {\bf Select}[{\bf Subsets}[{\bf Table}[i, \{i, k\}]], 0 <{\bf  Length}[\#] < k \&] }$\\
\end{tabular}
}

\medskip

{\peq
\begin{tabular}{l}
${\rm {\bf HS}[m\_] := {\bf Module}[\{pos, l, n = 0, k = 1\}, pos = ps[{\bf Length}[m]]; l =
{\bf Length}[pos];}$\\
\hskip 1cm ${\rm {\bf  Do}[}$\\
\hskip 1cm ${\rm {\bf  If}[{\bf HSC}[pos[[k]], m] == pos[[k]], n++], \{k, 1, l\}]; n]}$\\
\end{tabular}
}

\medskip

{\peq
\begin{tabular}{l}
${\rm {\bf CofinalQ}[m\_] := {\bf Module}[\{v = {\bf Table}[i, \{i, {\bf Length}[m]\}], r = {\bf True}\},}$\\
\hskip .5cm ${\rm {\bf Do}[r = r \&\& {\bf HSC}[\{i\}, m] == v, \{i, {\bf Length}[m]\}]; r] }$\\
\end{tabular}
}


\section*{acknowledgments}

The authors would like to thank Prof. Enrique Pardo for
his useful comments.


\end{document}